\documentclass[english]{amsart}
\usepackage[english]{babel}

\usepackage[hidelinks]{hyperref}

\usepackage{verbatim}
\usepackage{amsfonts,amssymb,mathtools}
\usepackage{amsmath}
\usepackage{bbm}
\usepackage[foot]{amsaddr}

\usepackage{todonotes}

\usepackage{graphicx,xcolor}
\definecolor{green}{RGB}{89,169,58}
\definecolor{red}{RGB}{224,61,42}
\definecolor{blue}{RGB}{63,78,181}

\graphicspath{{fig/}}
\usepackage{subfigure,tikz}
\usetikzlibrary{arrows}

\usepackage[style=numeric,maxcitenames=2,maxbibnames=10,doi=false,isbn=false,url=false,natbib=true,giveninits=true,hyperref,bibencoding=utf8,backend=biber]{biblatex}
\AtEveryBibitem{%
  \clearfield{note}%
  \clearfield{issn}%
  \clearlist{language}%
  }

\newcommand{\N}{\ensuremath{\mathbb{N}}}
\newcommand{\R}{\ensuremath{\mathbb{R}}}

\newcommand{\E}{\ensuremath{\mathbb{E}}}
\renewcommand{\P}{\ensuremath{\mathbb{P}}}

\newcommand{\ind}[1]{\ensuremath{\mathbbm{1}_{\left\{#1\right\}}}}
\newcommand{\diff}{\mathop{}\mathopen{}\mathrm{d}}
\newcommand{\cal}[1]{\ensuremath{\mathcal{#1}}}
\newcommand\croc[1]{\left\langle #1\right\rangle}
\newcommand\steq[1]{\stackrel{\text{\rm #1.}}{=}}

\def\eps{\varepsilon} 
\def\cadlag{c\`adl\`ag }
\def\card{{\rm Card}}
\newcommand\Pois[1]{{\rm Pois}\left[#1\right]}

\newtheorem{proposition}{Proposition}
\newtheorem{definition}[proposition]{Definition}
\newtheorem{lemma}[proposition]{Lemma}

\newtheorem{theorem}[proposition]{Theorem}

\newtheorem{corollary}[proposition]{Corollary}

\numberwithin{proposition}{section}

\title[Thermodynamic Limits of CRN]{Thermodynamic Limits of Stochastic Chemical Reaction Networks with Phosphorylation}

\date{\today}
\author[L. Laurence]{Lucie Laurence}
\email{Lucie.Laurence@unibe.ch}
\address[L. Laurence]{Institute of Mathematical Statistics and Actuarial Science, Department of Mathematics and Statistics, University of Bern, Alpeneggstrasse 22, 3012 Bern, Switzerland}
 \address[Ph. Robert]{INRIA Paris, 48, rue Barrault, CS 61534, 75647 Paris Cedex, France}
\author[Ph. Robert]{Philippe Robert}
\email{Philippe.Robert@inria.fr}
\urladdr{http://www-rocq.inria.fr/who/Philippe.Robert}

\begin{document}

\begin{abstract}
  In this paper we investigate the stability properties of a fundamental mechanism of biological cells called phosphorylation. The system is a chemical reaction network  (CRN) for which  a chemical species, {\em the substrate},  can be sequentially transformed into two phosphorylated forms, by the activity of two types of enzymes, one type for phosphorylation, the other for dephosphorylation. We investigate a stochastic representation of this model, under the mass action kinetics. The total mass of the substrate is fixed at $N$, while the total mass of enzymes scales proportionally to $N$. The asymptotic behavior, when $N$ is large, of the concentrations of all chemical species is studied. 
	
We investigate the possible {\em stable} subsets of chemical species for the kinetics of the law of mass action. A stable subset is such that, with a convenient initial state, the number of copies of the  species of this subset remains $O(1)$ on any finite time interval as $N$ gets large. The role of the twelve reaction rate constants, {\em the catalytic constants} of the CRN,  is investigated from this point of view. An averaging principle of the corresponding Markov process is established for several regimes of the CRN. It is shown in particular that there exists a regime with three equilibrium points, with two of them stable. The proofs of the results rely on stochastic calculus with Poisson processes, convenient couplings of subsets of coordinates of the Markov process, technical results on $M/M/\infty$ queues, and a stability analysis of a dynamical system in $\R_+^4$. 
\end{abstract}
\maketitle

 \vspace{-5mm}

\bigskip

\hrule

\vspace{-3mm}

\tableofcontents

\vspace{-1cm}

\hrule

\bigskip

\section{Introduction}
The model we investigate  is a chemical reaction network (CRN) for which some chemical species called {\em substrat}, has $d$ possible states ${S}_i$, $i{=}1$, \ldots, $d$. The transitions between these states are achieved via catalytic reactions involving {\em enzymes}. Enzymes of species ${A}$ operate  upward transformations ${ S}_i{\rightharpoonup}{ S}_{i+1}$, $1{\le}i{<}d$, in two steps, while enzymes of species ${ B}$ are for  downward transformations ${ S}_i{\rightharpoonup}{ S}_{i-1}$, $1{<}i{\le}d$.
These reactions are done via the creation of complexes with substrat and enzymes, ${ AS}_i$, $1{\le}i{<}d$ and  ${ BS}_i$, $1{<}i{\le}d$. See Section~\ref{PhosphoSec} for a quick presentation of the biological background of these important chemical reaction networks. The chemical reaction network analyzed in this paper is for $d{=}3$. See Graph~\eqref{CRN}.

\medskip
\hrule
\begin{equation}\label{CRN}
\begin{cases*}
        A{+}S_1\mathrel{\mathop{\xrightleftharpoons[\alpha_{1}^-]{\alpha_{1}}}} AS_1\stackrel{\lambda_1}{\rightharpoonup} A{+}S_2 \mathrel{\mathop{\xrightleftharpoons[\alpha_{2}^-]{\alpha_{2}}}}  AS_2\stackrel{\lambda_2}{\rightharpoonup}A{+}S_3\\
         B{+}S_1\stackrel{\mu_2}{\leftharpoonup} BS_2\mathrel{\mathop{\xrightleftharpoons[\beta_{2}]{\beta_{2}^-}}}  B{+}S_{2} \stackrel{\mu_1}{\leftharpoonup}BS_3\mathrel{\mathop{\xrightleftharpoons[\beta_{1}]{\beta_{1}^-}}}   B{+}S_3.
\end{cases*}
\end{equation}
\hrule

\medskip
The sets of chemical species ${\cal C}_P$ and of reaction rates ${\cal R}_P$ are given by 
\begin{align}
{\cal C}_P&\steq{def} \{  { S}_1, { S}_2,{ S}_3, { A}, { AS}_1, { AS}_2,{ B}, { BS}_2, { BS}_3\},\label{CS}\\
{\cal R}_P&\steq{def} \{ \alpha_1,\alpha_1^-,\alpha_2,\alpha_2^-,\beta_1,\beta_1^-,\beta_2,\beta_2^-,\lambda_1,\lambda_2,\mu_1,\mu_2\}.\label{RR}
\end{align}
Using the biological representation of these systems,
\begin{enumerate}
\item the chemical species ${ S}_2$ and ${ S}_3$  are the {\em phosphorylated} versions of species ${ S}_1$, the {\em substrat}.
\item The complex ${ AS_1}$, resp. ${ AS_2}$, is the species formed by an enzyme of type ${ A}$ bound to a macromolecule of ${ S}_1$, resp. ${ S}_2$, and similarly for enzymes of type ${ B}$ with ${ S}_2$ and  ${ S}_3$.
\end{enumerate}

\subsection{Phosphorylation Mechanisms}\label{PhosphoSec}
In a biological context, the substrat is a protein which has several sites potentially accepting an additional phosphate group. This addition is done via a catalytic reaction involving  an enzyme called {\em kinase}. The reverse operation, the removal of a phosphate group, is done with an other enzyme called {\em phosphatase}. This is a very important and general mechanism for a large class of eukaryotic cells, where 30\% of proteins are thought to be phosphorylated. In the human genome, there are more than 1000  types of protein kinases and 500  types of protein phosphatases, and the number of possible phosphorylation sites of some proteins can be over 100. See~\citet{Krebs1993}, \citet{Cohen2000}, or~\citet{Thomson2009}.

A protein with one or more additional phosphate group is a {\em phosphorylated version} of the protein. A phosphorylated protein will have in general different functional properties within the cell from the original protein.  Depending on the concentrations of kinases and phosphatases and on the environment of the cell,  there will be  different populations of phosphorylated versions of given protein and therefore different global functional properties.  Phosphorylation can be seen as a way of getting multiple functional agents via the production of one type of protein, without resorting to the costly production of numerous dedicated proteins. The different populations of phosphorylated proteins are a kind of  encoding of a protein function in a cellular system. See~\citet{Gunawardena2005}. 

The modifications of the phosphorylation sites of a protein, with either kinases or phosphatases,  may be achieved in a sequential way,  or in a random way. If the enzyme
is released after every modification, the mechanism is called {\em distributive}, whereas if the enzyme remains bound it is referred
to as {\em processive}. For a more detailed description of the terminology, see~\citet{Suwanmajo}, \citet{Salazar2009} or~\citet{Conradi2018}.  By combining these possibilities, this gives overall a quite large number of possible phosphorylation systems.  The CRN investigated in this paper can be described as a {\em dual distributive sequential phosphorylation system}.

\subsection{Mathematical Models}
They describe the time evolution of its state, a vector $v$ associated with the nine chemical species,
\[
v{=}(x_1,x_2,x_3,u_A,a_1,a_2,u_B,b_1,b_2).
\]
The coordinates of $v$ denote the concentration, or the number of copies,  of the species
\begin{itemize}
	\item $x_j$, for  ${ S}_j$, $j{\in}\{1, 2, 3\}$;
	\item $u_A$, resp. $u_B$,  for  ${ A}$, resp. ${ B}$;
	\item $a_{1}$, resp. $a_2$, for ${ AS}_{1}$, resp. ${ AS}_{2}$;
	\item $b_{1}$, resp. $b_2$, for ${ BS}_{3}$, resp. ${ BS}_{2}$.
\end{itemize}
Since there is no {\em ex-nihilo} creation or  definitive destruction of chemical species, the following conservation relations 
\begin{equation}\label{eqCons}
\begin{cases}
x_1+x_2+x_3+a_1+a_2+b_1+b_2=M_S,\\
E=u_A{+}a_1{+}a_2 \quad \text{and} \quad F=u_B{+}b_1{+}b_2,
\end{cases}
\end{equation}
hold, where
\begin{itemize}
\item $M_S$ is the total mass of substrat (phosphorylated/bound with enzymes or not);
\item $E$, resp. $F$,  is the total mass of enzymes of species ${ A}$, resp. ${ B}$.
\end{itemize}
The quantities $M_S$, $E$ and $F$ are assumed to be fixed. 

The kinetics considered for these networks is the classical {\em law of mass action}. See~\citet{guldberg1864studies} and~\citet{Voit2015} for example. 

\subsubsection{Deterministic Models of CRNs with Phosphorylation}\label{DetSec}

In this case, the vector of the concentrations of the chemical species is the solution of a dynamical system in $\R_+^9$ satisfying Relations~\eqref{eqCons}. It describes the time evolution of the CRN. Due to the kinetics of the law of mass action,  this is a set of nine ordinary differential equations with polynomial terms. We pick two examples of them to illustrate the law of mass action in this context,
\[
\begin{cases}
\dot{a}_1(t)&=\alpha_1 u_A(t)x_1(t)-(\lambda_1{+}\alpha_1^-)a_1(t),\\
\dot{x}_2(t)&=\lambda_1a_1(t)+\alpha_2^-a_2(t)-\alpha_2u_A(t)x_2(t) +\mu_1b_1(t)+\beta_2^-b_2(t)-\beta_2u_B(t)x_2(t).
\end{cases}
\]
This is the model, or a close variant of it, investigated in~\citet{Conradi2024,Feliu2022,Feliu2024}.

With the conservation relations~\eqref{eqCons}, the ODE in $\R_+^9$ can be reduced to an ODE in $\R_+^6$. Still, this is a complex dynamical system.

\subsubsection{Stochastic Models of CRNs with Phosphorylation}\label{StocSec}

With stochastic models, the state of the CRN is given by the number of copies of the nine chemical species. The time evolution is described by a Markov jump process with values in $\N^9$.  The law of mass action defines its infinitesimal generator and the $Q$-matrix of jump rates. For example, in state
$v{=}(x_1,x_2,x_3,u_A,a_1,a_2,u_B,b_1,b_2)$, for the two reactions with rates $\alpha_1$ and $\alpha_1^-$ in the CRN~\eqref{CRN}, the corresponding transitions are given by:
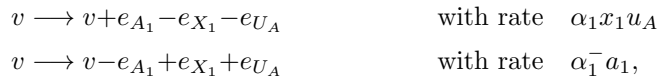
\begin{align*}
  v&\longrightarrow v{+}e_{A_1}{-}e_{X_1}{-}e_{U_A} &\text{with rate}\quad & \alpha_1x_1u_A\\
  v&\longrightarrow v{-}e_{A_1}{+}e_{X_1}{+}e_{U_A} &\text{with rate}\quad & \alpha_1^-a_1,
\end{align*}
where $e_z$ is the unit vector associated to the index of the chemical species $z$ (with the slight abuse of notation, we use $\{X_1, X_2, X_3, U_A, A_1, A_2, U_B, B_1, B_2\}$ for $\cal{C}_P=\{S_1, S_2, S_3, A, AS_1, AS_2, B, BS_1, BS_2\}$, to have a clear correspondence between the process and the coordinate). See Section~\ref{ModelSec}. Similarly to the deterministic model, with the conservation relations~\eqref{eqCons}, there is a possible Markovian description in dimension six. 

\subsection*{Objectives of Mathematical Models}
Because of the biological background, see Section~\ref{PhosphoSec},  the main variables of interest are the concentrations/numbers of copies $(x_1,x_2,x_3)$ of the three phosphorylated versions of substrat given by the chemical species ${ S}_1$, ${ S}_2$ and ${ S}_3$.

For deterministic CRNs  a significant part of the mathematical literature has been devoted to
\begin{enumerate}
\item The conditions on the reaction rates and the topology for the existence of multiple equilibrium points of the solution of the ODE. See, for example, \citet{Barab}, \citet{Bazzani2012}, \citet{Feliu2022}, \citet{Holstein}, \citet{Flockerzi2014}, \citet{Markevich2004}, \citet{Salazar2009}, or \citet{Thomson2009}.
\item The existence of oscillatory behavior, i.e. when the dynamical system exhibits a Hopf-bifurcation, see \citet{Conradi2019,Conradi2024} \citet{Feliu2024}.
\end{enumerate}
Some approximations such as the quasi-steady state approximation have been used to further simplify these ODES. It consists essentially of speeding-up some chemical reactions which may justify the removal of chemical species ${ AS}_1$, ${ AS}_2$, ${ BS}_3$ and ${ BS}_2$. The resulting ODE is not anymore a polynomial dynamical system. It is expressed in terms of Hill functions: the ``Michaelis-Menten kinetics''. The validity of such approximation has to be discussed in practice and also justified from a theoretical point of view. See Salazar and H\"ofer~\cite{Salazar2009,Salazar2006}, and~\citet{Gunawardena2014}.

There are few studies of stochastic models of CRNs with phosphorylation up to now. They are mainly devoted to simulations of specific systems, see~\citet{Steijaert}, \citet{Lopez2018}, or~\citet{Zippo2025}.

\subsection{Thermodynamic Limits}
In this paper we study a stochastic model of the CRN~\eqref{CRN} when the total number of copies of substrat $M_S{=}N$, interpreted as a {\em volume} variable,  is large and  
the total number of copies $E_N$, resp. $F_N$, of  enzymes of type ${ A}$, resp. ${ B}$ are proportional to $N$,
\begin{equation}\label{ScalIntr}
\lim_{N\to+\infty}\left(\frac{E_N}{N},\frac{F_N}{N}\right)=(e,f) \in (0, +\infty)^2.
\end{equation}
This is a reasonable biological assumption, see \citet{Albe} and~\citet{Thomson2009} for example. 

The full state descriptor  $(V_N(t))$ of the time evolution of our CRN has values in $\N^9$,
\[
(V_N(t))=(X_1^N(t),X_2^N(t),X_3^N(t),U_A^N(t),A_1^N(t),A_2^N(t),U_B^N(t),B_1^N(t),B_2^N(t)).
\]

The deterministic picture of the CRN in Section~\ref{DetSec} is {\em not} an asymptotic limit of the process of concentration $(V_N(t)/N)$ as $N$ gets large.
The main reason is that there are reactions of the order of $N^2$, for the formation of complexes ${ AS}_i$, ${ BS}_i$, $i{=}1$, $2$, while all the other reactions would be of the order $N$. Such a limiting result would hold but with the following change: For $i{=}1$, $2$, the reaction rates $\alpha_i$, resp. $\beta_i$,  have to be replaced by $\alpha_i/N$, resp. $\beta_i/N$. In this way, some chemical reactions are slowed down.  This is changing in fact the very structure of the CRN when $N$ is large. 

This is a common feature of deterministic CRNs viewed as an asymptotic description of discrete models of chemical reactions: The total number of copies of every chemical species is ``large'', of the order of the volume of the reaction.  See~\citet{Mozgunov} and~\citet{LR23} for example. In our paper, we consider a fixed topology of a CRN with phosphorylation,  the reaction rates do not depend on $N$. 

\subsection{Identification of Stable Regimes}
From the point of view of thermodynamic limits of CRNs, a key difference between deterministic and stochastic models is that  deterministic models implicitly assume that, at the microscopic level,  all chemical species are of the same order of magnitude $N$. This is not necessarily true for stochastic models. This is an important motivation of stochastic models for biological systems where the number of copies of some  chemical species with high activity  may be small.

In a stochastic context,  it may happen that if some species are of the order of $N$, others are ``$O(1)$'' random variables. 
The main contribution of our paper is of investigating the possible ``correct'' orders of magnitude for the coordinates of $(V_N(t))$ of the $CRN$ with phosphorylation. To the best of our knowledge, these important aspects do not seem to have been investigated for CRNs with phosphorylation. We give a  heuristic, rough,  formulation of the results we investigate.

If $H$ is a subset of ${\cal C}_P$ and  $|H|{=}\card(H)$, we denote $(V^H_N(t))$ as the vector  in $\N^{|H|}$  with the coordinates of $(V_N(t))$ whose indices are in $H$. $H^c$ denotes ${\cal C}_P{\setminus}H$. 

Throughout the paper,  $H$ will stand for the subset of ${\cal S}_P$ indicating the indices of the coordinates which are $O(1)$, and the other coordinates,  in $H^c$, are  of the order of $N$.  See  Definition~\ref{Order-Def} for a rigorous formulation. 
The regime associated to a subset $H$ of ${\cal C}_P$ is said to be {\em stable} if the following properties are satisfied: Under the assumptions that,
\begin{itemize}
\item[---]
  the values of the sequence $(V^{H}_N(0))$ are in a fixed finite subset of $\N^{|H|}$;
\item[---]
  The sequence $(V^{H^c}_N(0)/N)$  is converging to an arbitrary element $y_0$ of a non-empty open subset of $(0,{+}\infty)^{9-|H|}$,
\end{itemize}
then, the coordinates of $(V^{H}_N(t))$ are $O(1)$  and, for the convergence in distribution,  the relation
\begin{equation}\label{Dynv}
\lim_{N\to+\infty} \left(\frac{V^{H^c}_N(t)}{N}\right)=(v(t))
\end{equation}
holds where $(v(t))$ is a non-trivial dynamical system starting at $y_0$. 

The associated dynamical system $(v(t))$, solution of a deterministic ODE, is in a state space of dimension $|H^c|$  strictly less than $9$,  the dimension of the ODE in the deterministic case defined in Section~\ref{DetSec}. The downside of the stochastic case is that the associated ODE is not anymore expressed in terms of polynomials in the state variable, but with rational functions. Note that rational functions also arise in some deterministic models in the literature, particularly when quasi-steady state approximations are used, leading to the so-called Michaelis-Menten kinetics, as in~\citet{Markevich2004}.

In our paper, depending on the parameters of the system, the sets of chemical species H that {\em can be} stable are
\[
\{{ S}_2,{ S}_3, { A},  { AS}_2, { BS}_3\},\quad
\{  { S}_1, { S}_2,  { AS}_1,{ B},  { BS}_2\},\quad 
\{   { A}, { B}\},\quad
\{  { S}_1, { S}_2,{ S}_3\}.
\]
It should be noted that, depending on the conditions on $e$ and $f$ of Relations~\eqref{ScalIntr} and on the set ${\cal R}_P$ of reaction rates, such a convergence in distribution to a dynamical system $(v(t))$ is valid for several of the above sets $H$, but on a finite time interval $[0,t_0)$, with $t_0{>}0$. The difficulty is to find the conditions for which the subset $H$ is stable, i.e. when the convergence holds for any finite interval $[0, T]$, $T{>}0$.

 The interplay between these ``large'' and ``small'' coordinates is at the core of the difficulty of these mathematical models. This is essentially related to the proof of a convenient averaging principle, see~\citet{Kurtz1992}, and also to the proof of a stability result of some of the equilibrium points of the deterministic dynamical system $(v(t))$ of Relation~\eqref{Dynv}.

 \subsection*{Contributions}
 In the study of deterministic networks, a significant effort has been devoted to find conditions on the twelve reactions rates, the {\em catalytic constants},  to have multiple  stable points, or an oscillating behavior. See for example~\citet{Conradi2014}, \citet{Conradi2019}  and~\citet{Conradi2024}. Due to the number of parameters, this is a complicated issue, there are many open questions in this domain. 

 The role of the twelve reaction rates of ${\cal R}_P$  and of the values of the total concentrations $e$ and $f$ of enzymes of species ${ A}$ and ${ B}$  is discussed in our setting from the point of view of the stability property described above. We now give a sketch of our results. 

 \subsection{Condition $e{+}f{<}1$} There are less enzymes than substrat. See Section~\ref{M1Sec} and Appendix~\ref{M1RSec}.  This is a common biological assumption, see~\citet{Albe}. In this case, the stability properties of the CRN depend only on the parameters $e$, $f$ and the reaction rates $\lambda_1$, $\lambda_2$, $\mu_1$ and $\mu_2$. See Theorem~\ref{ThM1M} and Proposition~\ref{StabM1M} for example. These reaction rates are the rates of break-up of the complexes ${ AS}_1$, ${ AS}_2$, ${ BS}_3$ and ${ BS}_2$ to produce the next phosphorylated (or un-phosphorylated) version of the substrat with the species ${ A}$ or ${ B}$. Table~\ref{Tab1} below summarizes the results on the stability of three possible subsets $H$ of indices.

 \begin{table}[h]
\begin{tabular}{|c|c|c|c|c|}
  \hline Set $H$ of $O(1)$ species   &$\{{ S}_2,{ S}_3, { A},  { AS}_2, { BS}_3\}$&$\{  { S}_1, { S}_2,  { AS}_1,{ B},  { BS}_2\}$&$\{   { A}, { B}\}$\\
  \hline $\lambda_1e<\mu_2f$, $\mu_1f>\lambda_2e$& stable & $\emptyset$ & $\emptyset$ \\
  \hline $\lambda_1e>\mu_2f$, $\mu_1f<\lambda_2e$&$\emptyset$&stable & $\emptyset$ \\
  \hline $\lambda_1e>\mu_2f$, $\mu_1f>\lambda_2e$& $\emptyset$ &$\emptyset$& stable${}^*$\\
  \hline $\lambda_1e<\mu_2f$, $\mu_1f<\lambda_2e$&stable&stable &unstable${}^*$ \\
  \hline
\end{tabular}
\vspace{4mm}

\caption{Stability Properties of the CRN when $e{+}f{<}1$}\label{Tab1}
 \end{table}

 The symbol $\emptyset$ is used when there does not exist a positive equilibrium inside the state space of the CRN for the  associated dynamical system $(v(t))$ of Relation~\eqref{Dynv}\footnote{The ${}^*$ in the table indicates that the result has been obtained when  all $\alpha_{i}^-$, $\beta_{i}^{-}$, $i{=}1$, $2$, are either $0$ (circular topology) or $1$}. In this later case, this amounts to the fact that the subset $H$ of $O(1)$ species is not the ``right'' one. Some coordinates have to change their order of magnitude. 

\begin{enumerate}
 \item Under the conditions $\lambda_1e{-}\mu_2f{<}0$, $H{=}\{{ S}_2,{ S}_3, { A},  { AS}_2, { BS}_3\}$, is stable.  In this regime, the weight of the system is ``on the left'' of the CRN, $X_1, A_1, B_2$ and $U_B$ are $O(N)$, and the species ``on the right'' of the CRN are $O(1)$.
  \item Under the conditions $\mu_1 f{-}\lambda_2e{<}0$, $H{=}\{  { S}_1, { S}_2,  { AS}_1,{ B},  { BS}_2\}$, is stable.  This is the symmetrical regime of the previous one: the weight of the system is ``on the right'', $X_3, B_1, A_2$ and $U_A$ are $O(N)$, and the species ``on the left'' of the CRN are $O(1)$.
	\item The set $H{=}\{ { A}, { B}\}$ corresponds to an interesting regime, where the number of free enzymes are $O(1)$: almost all species ${ A}$ and ${ B}$ are bound to substrat. Under the condition $\lambda_1e{-}\mu_2f{>}0$ and $\mu_1 f{-}\lambda_2e{<}0$, it is stable, whereas when $\lambda_1e{-}\mu_2f{>}0$ and $\mu_1 f{-}\lambda_2e{<}0$ hold, it is shown to be  unstable. This has been proved when the values of $\alpha_{i}^-$, $\beta_{i}^{-}$, $i{=}1$, $2$ are all  $0$ (circular topology) or $1$. This restriction is due to the limitation of symbolic computations used in Section~\ref{M1MSec}. We conjecture that the property also holds for arbitrary  $\alpha_{i}^-$, $\beta_{i}^{-}$, $i{=}1$, $2$
 \end{enumerate}

The condition $\lambda_1\mu_1{-}\lambda_2\mu_2{>}0$ of ~\citet{Conradi2024} plays also a role in our classification. This condition holds in the third row of Table~\ref{Tab1}.  Interestingly, it has been shown in this reference that, under this condition, there is a possibility of an oscillating regime for the deterministic version of the CRN. The two mathematical models of CRNs being quite different, the connection with our results is not clear.

It should be noted that we did not study the (hard) problem of determining the {\em basin of attractions} of stable equilibrium points, neither the eventual behavior of the sample paths starting in the neighborhood of an unstable equilibrium point.

\subsection{Condition $e{+}f{>}1$}
The initial motivation of this work was, for biological reasons, the case $e{+}f{<}1$.  See~\citet{Albe}. To have an overview of the general properties of these mathematical models, we have also investigated the set of possible equilibrium points (not their stability properties) under the assumption $e{+}f{>}1$. See Section~\ref{SatSec}.

Since there is, a priori, an oversupply of enzymes, the regime investigated is when the number of copies of the species ${ S}_1$, ${ S}_2$ and ${ S}_3$ are $O(1)$. The situation has some similarities with the previous case, but it is more delicate for some aspects. The sign of $\lambda_1\mu_1{-}\lambda_2\mu_2$ plays, directly this time,  an important role for the existence and uniqueness of an equilibrium point. A multi-stability property nevertheless may occur  depending on the parameters $e$, $f$ and $\lambda_1$, $\lambda_2$, $\mu_1$ and $\mu_2$ as before, but also  on the parameters $\alpha_2$ and $\beta_2$. See Proposition~\ref{M2ExFP}. 
 
In both cases  $e{+}f{<}1$ and  $e{+}f{>}1$, the reaction rates $\alpha_1$, $\alpha_1^-$, $\alpha_2^-$, $\beta_1$, $\beta_1^-$ and $\beta_2^-$ do not seem to have an impact on the equilibrium properties of these networks, even if some  equilibrium characteristics, like the limits of occupation measures,  depend on these parameters.

\subsection{A Technical Overview}
The convergence results of the paper are related to the proof of an averaging principle for a sequence of Markov processes. See the general references~\citet{Khasminskii1}, \citet{Papanicolaou} or Chapter~7 of~\citet{Freidlin}. In the context of jump process, a useful, comprehensive, approach is presented in~\citet{Kurtz1992}. This is a classical ingredient of the analysis of stochastic CRNs, see \citet{BallKurtz}, \citet{KangKurtz}, \citet{Kim2017},  and~\citet{LR24-2,LR23,LR24}.

Classically, in this setting, the coordinates of the Markov process are partitioned into two subsets, one for ``slow'' processes and the others are fast process. Note that this is a rough description since this separation of variables may be a little more subtle sometimes. The standard approach consists in
\begin{enumerate}
\item Proving  tightness properties of the associated sequence of occupation measures of the fast processes;
  \item Proving the tightness of the slow processes. 
\end{enumerate}
Step~(1) can be quite complicated to achieve, especially when the dimension of the state space of fast processes is greater than $2$.

For our CRN, the main initial difficulty is of proving that for a subset $H$, if the initial state $V_N(0)$ is such that the coordinates of $V_N^H(0)$ are $O(1)$ and the coordinates of $V_N^{H^c}(0)$ are $O(N)$, as $N$ gets large, then this property is valid on a non-empty time interval. As it is well-known, explosion mechanisms of CRNs may change the orders of magnitude of some coordinates very quickly. The fact that the dimension of the slow variables of the order of $N$ is at least 4 complicates the ``control'' of the coordinates of the Markov process. In~\citet{FRZ25}, also in a multi-dimensional setting, the control of the $O(1)$ processes is achieved via a different approach,  with upper-bounds on the fluctuations of fast process on finite time intervals. We review several technical steps of our approach. 

\begin{enumerate}
\item[a)] {\sc Couplings}\\
We  introduce several couplings of subsets of the coordinates. In order to use classical stochastic calculus with Poisson processes including the other coordinates, the couplings have to be defined carefully on the initial probability space. See the proof of Theorem~\ref{ThM1L}. These couplings use networks of $M/M/\infty$ queues. Several technical results are used in this domain, Lemma~\ref{MMILower} and Proposition~\ref{MMIHit}, or derived, like Proposition~\ref{MMNet}. See Section~\ref{MMISec}.

Our convergence results are for the sequence of the scaled slow variables to the solution of a deterministic ODE. They also give the explicit expression of the limiting value of the occupation measures of the $O(1)$ variables. In two regimes of this CRN, this limit is a {\em random } measure and not a deterministic one as it is classically the case in most of averaging principle results. 

\item [b)] {\sc Slow $O(1)$ variables}. \\
In most of averaging principles proved for CRNs, the $O(1)$ variables, if any, are fast variables. This is not the case for two regimes of these CRNs. 

  In Section~\ref{M1LSec}  three $O(1)$ coordinates $(A_2^N(t),X_3^N(t),B_1^N(t))$ of $(V_N^H(t))$ are slow, the transition rate of their positive jumps are bounded.They have not  been added to the slow variables but kept in the occupation measure for the main reason that,  if there is likely a convergence for the Skorohod topology of $(A_2^N(t),B_1^N(t))$,  it does not hold for $(X_3^N(t))$ (towards $(0)$), although it would probably hold for a weaker topology. For simplicity, to avoid annoying technicalities with limited interest, we have chosen to consider only the convergence of their occupation measures. There is also a (simpler) case of this situation in  Proposition~6 of~\citet{LR24}.

  It should be noted that the invariant distribution of fast variables do not have necessarily a product form. See Section~\ref{M1MSec}.
  
\item[c)] {\sc Stability Properties}.\\
  The stability properties of the equilibrium of the dynamical systems $(v(t))$ of Relation~\eqref{Dynv} are easy to establish for the regimes of Section~\ref{M1LSec}  and of Section~\ref{M1RSec} of the appendix. For one of the regimes of  Section~\ref{M1MSec}, the dynamical system is an ODE in dimension~4 with one equilibrium point $v_*$. Its stability properties are more delicate to handle. The numerator of the characteristic polynomial $P_*$ of the Jacobian matrix at the point $v_*$ has degree~4, but some of its coefficients are polynomials of degree~37  with respect to the parameters of the CRN: The twelve reaction rates and $e$ and $f$ of Relation~\eqref{ScalIntr}. 

We use a classical criterion to show that all roots of a polynomial have or not a negative real part: The Routh-Hurwitz Criterion, a (nice) result of the 19th century (!). One has to check the sign of five constants expressed with the coefficients of $P_*$, under the conditions of the regime of Section~\ref{M1MSec}  defined by three inequalities for reaction rates and  $e$ and $f$. Note that a variant of the Routh-Hurwitz Criterion has been used in~\citet{Conradi2019} to prove a Hopf-bifurcation of  a class of deterministic CRNs with phosphorylation.

Given the complexity of the coefficients of the polynomial $P_*$, in view of the  Routh-Hurwitz Criterion, the main ingredient of the proof of our main stability result, Proposition~\ref{StabM1M}, uses a convenient parametrization of four constants (two reaction rates and $e$ and $f$) of the CRN with four variables. The three inequalities defining the regime are equivalent to the fact that these four variables are positive. The next step consists in getting the expression of Routh-Hurwitz constants with the help of the symbolic computation language Maple\texttrademark~\cite{Maple}. Because of the parametrization, it turns out that the sign of each of them is then quite easy to obtain.  
\end{enumerate}

\subsection{Organization of the Paper}
Section~\ref{ModelSec} introduces the general notations and definitions and results on networks of $M/M/\infty$ queues are presented. 

Section~\ref{M1Sec}  investigates the case $e{+}f{<}1$, for $e$ and $f$ of Relation~\eqref{ScalIntr}. The total mass of enzymes is strictly less than the total mass of substrat. Section~\ref{M1LSec} analyses the regime $H{=}\{{ S}_2,{ S}_3, { A},  { AS_2}, { BS_3}\}$ and
in Section~\ref{M1MSec},  $H{=}\{{ A}, { B}\}$. The case $H{=}\{{ S}_1, { S}_2,  { AS_1},{ B},  { BS_2}\}$, similar to the regime of Section~\ref{M1LSec}, is briefly presented in Section~\ref{M1RSec} of the Appendix. 

Section~\ref{M1Sec} considers the case $e{+}f{>}1$. The possible equilibrium points of the regime $H{=}\{{ S}_1,{ S}_2,{ S}_3\}$ are investigated.

\section{Stochastic Model}\label{ModelSec}
\subsection{Notations and Definitions}
If $E$ is a locally compact space ${\cal B}(E)$ is the set of Borelian subsets of $E$ and ${\cal C}(E)$, resp. ${\cal C}_c(E)$, denotes the set real-valued continuous functions on $E$, resp. continuous functions with compact support on $E$. The set of Radon measures on $E$ is denoted by ${\cal M}(E)$, it is endowed with the topology of weak convergence. Finally, ${\cal M}_P(E)$ denotes the subset of ${\cal M}(E)$ of probability distributions on $E$. See~\citet{Rudin}. 

A \cadlag process $(X(t))$ is a stochastic process whose sample paths are almost surely continuous on the right with left limit $X(s{-})$ at every point $s{>}0$. 

\subsection*{Probability space}
For $a{>}0$, $\Pois{a}$ denotes the Poisson distribution on $\N$ with parameter $a$. 

It is assumed that, for any element $\kappa$ of the set of reaction rates ${\cal R}_P$, see Relation~\eqref{RR}, there is a Poisson process  ${\cal P}_\kappa$ on $\R_+^2$ with intensity measure $\diff s{\otimes}\diff t$ on $\R_+^2$. See Chapter~1 of~\citet{Robert}. The Poisson processes ${\cal P}_\kappa$, $\kappa{\in}{\cal R}_P$, are independent.  

For $a{\ge}0$ and $\kappa{\in}{\cal R}_P$, we will use  the differential notation, 
\[
  {\cal P}_\kappa((0,a),\diff t) =\int_{\R_+}\ind{s{\le}a}{\cal P}_\kappa(\diff s,\diff t), 
\]
i.e., for $f{\in}{\cal C}_c(\R_+^2)$,
\[
\croc{{\cal P}_\kappa((0,a), \cdot),f}=\int_{\R_+^2}\ind{s{\le}a}f(t){\cal P}_\kappa(\diff s,\diff t).
\]
The associated filtration is $({\cal F}_t)$, where, for $t{\ge}0$,  ${\cal F}_t$ is the completed  $\sigma$-field generated by the set of random variables
\begin{equation}\label{SDEFilt}
\left\{{\cal P}_\kappa(A{\times}[0,s)), \kappa{\in}{\cal R}_P, s{\le}t, A{\in}{\cal B}(\R_+)\right\}. 
\end{equation}
Throughout the paper, the notions of adapted, optional processes, of martingale and stopping time are implicitly with respect to this filtration. See~\citet{Rogers2} for general definitions  and results of stochastic calculus. 

\subsection{The Markov Process}
The complete state descriptor of the CRN~\eqref{CRN} is given by the vector $(x_1,x_2,x_3,u_A,a_1,a_2,u_B,b_1,b_2)$. The conservation relations~\eqref{eqCons} give three relations which, in principle,  reduce the dimension of the state space from nine to six. For the simplicity of presentation, of stable sets in particular, we will nevertheless represent the state space ${\cal S}_P$ of our process as a subset of $\N^9$, 
\[
{\cal S}_P\steq{def} \left\{v=(x_1,x_2,x_3,u_a,a_1,a_2,u_b,b_1,b_2){\in}\N^6: \substack{\displaystyle x_1{+}x_2{+}x_3{+}a_1{+}a_2{+}b_1{+}b_2{=}M_S\\\ \\\displaystyle u_a{+}a_1{+}a_2{=} E,\, u_b{+}b_1{+}b_2{=}F}\right\}.
\]
For $v{=}(x_1,x_2,x_3,u_a,a_1,a_2,u_b,b_1,b_2){\in}{\cal S}_P$, we will also use  the representation
\[
v=(v_{z}, z{\in}{\cal C}_P),
\]
so that $v_{X_i}{=}x_i$, $i{=}1$, $2$, $3$ and $v_{U_A}{=}u_A$, $v_{U_B}{=}u_B$, and, for $i{=}1$, $2$, $v_{A_i}{=}a_i$, $v_{B_{i}}{=}b_{i}$, with the slight abuse of notation ${A_i}{=}{AS_i}$ and $B_{4-i}{=}B_{i}{=}BS_{4-i}$, $i{=}1$, $2$.

\subsection*{Law of Mass Action}\label{MassSec}
The kinetics of CRN with phosphorylation is assumed to be the law of mass action. For the associated Markov process, starting from $v{=}(x_1,x_2,x_3,u_A,a_1,a_2,u_B,b_1,b_2){\in}{\cal S}_P$, the possible jumps with the corresponding transition rates are 
\[
    \begin{cases}
\scriptstyle{e_{A_1}{-}e_{U_A}{-}e_{X_1}}, &\scriptstyle{\alpha_1x_1u_A},\\
\scriptstyle{{-}e_{A_1}{+}e_{U_A}{+}e_{X_1}}, & \scriptstyle{\alpha_1^-a_1},\\
\scriptstyle{e_{B_1}{-}e_{U_B}{-}e_{X_3}}, & \scriptstyle{\beta_1x_3u_B},\\
\scriptstyle{{-}e_{B_1}{+}e_{U_B}{+}e_{X_3}}, & \scriptstyle{\beta_1^-b_1},
\end{cases}
\begin{cases}
\scriptstyle{e_{A_2}{-}e_{U_A}{-}e_{X_2}}, &\scriptstyle{\alpha_2x_2u_A},\\
\scriptstyle{{-}e_{A_2}{+}e_{U_A}{+}e_{X_2}}, &\scriptstyle{\alpha_2^-a_2},\\
\scriptstyle{e_{B_2}{-}e_{U_B}{-}e_{X_2}}, &\scriptstyle{\beta_2x_2u_B},\\
\scriptstyle{{-}e_{B_2}{+}e_{U_B}{+}e_{X_2}}, &\scriptstyle{\beta_2^-b_2},
\end{cases}
\begin{cases}
\scriptstyle{{-}e_{A_1}{+}e_{U_A}{+}e_{X_2}}, &\scriptstyle{\lambda_1a_1},\\
\scriptstyle{{-}e_{A_2}{+}e_{U_A}{+}e_{X_3}}, &\scriptstyle{\lambda_2a_2},\\
\scriptstyle{{-}e_{B_1}{+}e_{U_B}{+}e_{X_2}}, &\scriptstyle{\mu_1b_1},\\
\scriptstyle{{-}e_{B_2}{+}e_{U_B}{+}e_{X_1}}, &\scriptstyle{\mu_2b_2},
\end{cases}
\]
with the notations $(e_z, z{\in}{\cal C}_p)$ of Section~\ref{StocSec}.
\begin{multline*}
(V(t))\steq{def} (X_1(t), X_2(t), X_3(t), U_A(t), A_1(t), A_2(t), U_B(t), B_1(t), B_2(t))\\=(V_{z}(t), z{\in}\cal{C}_P)
\end{multline*}

\subsection*{Stochastic Differential Equations (SDEs)}
The coordinates of the process $(V(t))$ can be represented with the solution $(X_1(t),X_3(t),A_1(t),A_2(t),B_1(t),B_2(t))$ of the set of SDEs,   for $t{\ge}0$ and $i{\in}\{1,2\}$,
\begin{equation}\label{SDE}
	\begin{cases}
	  \diff X_1(t)&= \cal{P}_{\alpha_1^-}((0, \alpha_1^-A_1(t{-})), \diff t) {+}\cal{P}_{\mu_2}((0, \mu_2B_2(t{-})), \diff t)\\
                &\hspace{4cm}{-}\cal{P}_{\alpha_1}((0,\alpha_1U_A(t{-})X_1(t{-})), \diff t),\\ 
		\diff X_3(t)&= \cal{P}_{\beta_1^-}((0,\beta_1B_1(t{-})), \diff t) {+}\cal{P}_{\lambda_2}((0,\lambda_2A_2(t{-})), \diff t)\\
                &\hspace{4cm}{-}\cal{P}_{\beta_1}((0,\beta_1U_B(t{-})X_3(t{-})), \diff t), \\
		\diff A_i(t)&= \cal{P}_{\alpha_i}((0,\alpha_iU_A(t{-})X_i(t{-})), \diff t)\\
                      &\hspace{2cm}{-}\cal{P}_{\alpha_i^-}((0,\alpha_iA_i(t{-})), \diff t){-}\cal{P}_{\lambda_i}((0,\lambda_iA_i(t{-})), \diff t),\\
		\diff B_i(t)&=\cal{P}_{\beta_i}((0,\beta_iU_B(t{-})X_{4{-}i}(t{-})), \diff t)\\
&\hspace{2cm}{-}\cal{P}_{\beta_i^-}((0,\beta_i^-B_i(t{-})), \diff t){-}\cal{P}_{\mu_i}((0,\mu_iB_i(t{-})), \diff t),\\
	\end{cases}       
\end{equation}
and the processes $(U_A(t))$, $(U_B(t))$ and $(X_2(t))$ are defined by the relations of mass conservation ,
\begin{align}
X_1(t)+X_2(t)+X_3(t)+A_1(t)+A_2(t)+B_1(t)+B_2(t)=M_S,\label{eqCons1}\\
U_A(t)=E{-}A_1(t){-}A_2(t) \text{ and } U_B(t)=F{-}B_1(t){-}B_2(t).\label{eqCons2}
\end{align}
See Appendix~A-1 of~\citet{LR23} for more details on this type of representation.

If $H$ is a subset  of ${\cal C}_P$, recall that
\[
(V^H(t))\steq{def} (V_z(t),z{\in}H),
\]
the random measure $\Lambda^H$ on $\R_+{\times}\N^{|H|}$ is the occupation measure associated to the process $(V^H(t))$,  the measure defined by, for $f{\in}{\cal C}_c(\R_+{\times}\N^{|H|})$,
\[
\croc{\Lambda^H,f}=\int_0^{+\infty} f\left(t,(V^H(t))\right)\diff t,
\]
and, for $T{>}0$, the occupation measure $\Lambda^{H,T}$ associated to the  process $(V^H(t))$ killed at $T$ is defined as
\begin{equation}\label{OccT}
\croc{\Lambda^{H,T},f}=\int_0^{T} f\left(t,(V^H(t))\right)\diff t.
\end{equation}
See~\citet{Dawson} for a general presentation of random measures.

\subsection{Thermodynamic Limits Framework}\label{ScalingC}
Throughout the paper it will be assumed that the total mass of substrat $M_S$ is $N$, the ``volume'' of the chemical reaction and the total masses $E_N$ and $F_N$ of the two enzyme species ${ A}$ and ${B}$ satisfy the relations
\begin{equation}\label{TLimits}
\lim_{N\to+\infty}\left(\frac{E_N}{N},\frac{F_N}{N}\right)=(e,f) \in (0, +\infty)^2.
\end{equation}
The corresponding process $(V(t))$ is denoted by
\[
(V_N(t))=(X^N_1(t), X^N_2(t), X^N_3(t),U_A^N(t),A^N_1(t),A^N_2(t),U_B^N(t),B^N_1(t),B^N_2(t)),
\]
and with similar notations,  for  $(V_N^H(t))$ and $(\Lambda_N^H)$, for $H{\subset}{\cal C}_P$. We define the scaled process
\begin{equation}\label{ScalNot}
\left(\overline{V}^H_{N}(t)\right)\steq{def}\left(\frac{V_z^N(t)}{N},z{\in}H\right).
\end{equation}
We will study the asymptotic behavior of this model for a convenient subset $H$ of indices. The scaling results obtained give the order of magnitude in $N$ of the different chemical species. A subset $H$ is said to be {\em stable} when the coordinates in $H$ stay $O(1)$ and the others stay of the order of $N$, on the whole time axis $\R_+$.  Mathematically, this is expressed rigorously by the following definition.

\begin{definition}\label{Order-Def}
A subset  $H{\subset}{\cal C}_P$  is said to be {\em stable} for the sequence of Markov processes $(V_N(t))$ if  there exists a non-empty open subset $O_P$ in the interior of $\R_+^{|H^c|}$ and a deterministic function $v{:}\R_+{\to}O_p$, such that, for any $v_0^{H^c}{\in}O_P$ and $v^H_0{\in}\N^{|H|}$, if  
\begin{equation}\label{CVIn}
\lim_{N\to+\infty} \left(\frac{V^{H^c}_{N}(0)}{N}\right)=v_0^{H^c}\text{ and } v_N^U(0){=}v^H_0, \forall N{\ge}1,
\end{equation}
then the convergence in distribution, 
\begin{equation}\label{eqCVM}
\lim_{N\to+\infty}  \left(\left(\overline{V}^{H^c}_{N}(t)\right), \Lambda^H_{N}\right)=((v(t)), \Lambda^H_\infty),
\end{equation}
holds, where $\Lambda^H_\infty$ is a random measure on $\R_+{\times}\N^{|H|}$, such that, almost surely,  $\Lambda^H_{\infty}(\diff t, \N^{|H|})$ is the Lebesgue measure on $\R_+$.
\end{definition}

The condition that the subset $O_P$  is in the interior of $\R_+^{|H^c|}$ ensures that the coordinates of $(V_N^{H^c}(t))$ are exactly of the order of $N$. The condition that $\Lambda^H_{\infty}(\diff t, \N^{|H|})$ is the Lebesgue measure on $\R_+$ implies that, for the occupation measure,  the coordinates of  $(V_N^{H}(t))$ remains in a finite neighborhood of $0{\in}\N^H$ with high probability. This is the rigorous formulation of the ``$O(1)$ property'' for $(V_N^{H}(t)$).

The convergence~\eqref{eqCVM} is equivalent to  the convergence in distribution
\begin{equation}\label{eqCVt0}
  \lim_{N\to+\infty}  \left(\left(\overline{V}^{H^c}_{N}(t), t{<}T\right), \Lambda^{H,T}_{N}\right)=((v(t),t{<}T), \Lambda^{H,T}_\infty), 
\end{equation}
for all $T{>}0$, where  $\Lambda^{H,T}_{N}$ is defined by Relation~\eqref{OccT}. 

The proof of the stability of $H$ is achieved in two steps. First prove that a convergence~\eqref{eqCVt0} holds for some $T{=}t_0{>}0$. The second step investigates the stability properties of an equilibrium of $(v(t))$ in $O_P$. The following intuitive proposition shows that this is enough to prove the stability of $H$. 
\begin{proposition}\label{StabProp}
For a subset $H$ of  ${\cal C}_P$, if there exist a non-empty open subset $O_P$ in the interior of $\R_+^{|H^c|}$ and $t_0{>}0$  such that,
  \begin{enumerate}
    \item For any $v_0^{H^c}{\in}O_P$ and $v^H_0{\in}\N^{|H|}$, if the initial conditions of $(V_N(t))$ satisfy Relation~\eqref{CVIn}, then the convergence in distribution 
\[
\lim_{N\to+\infty}  \left(\left(\overline{V}^{H^c}_{N}(t),t{<}t_0\right), \Lambda^{H,t_0}_{N}\right)=((v(t),t<t_0), \Lambda^{t_0}_\infty),
\]
holds, where $v{:}[0,t_0){\to}O_P$ is the solution of an ODE and  $\Lambda_\infty$ is a random measure on $\R_+{\times}\N^{|H|}$, such that, almost surely,  $\Lambda^{t_0}_{\infty}(\diff t, \N^{|H|})$ is the Lebesgue measure on $[0,t_0]$;
\item There exists a {\em stable} equilibrium point of the dynamical system $(v(t))$ in $O_P$;
  \end{enumerate}
then $H$ is stable for the sequence of Markov processes $(V_N(t))$
\end{proposition}
\begin{proof}
If $v_*$ is a stable point of $(v(t)$ in $O_P$, the stability property gives that  there exists an open neighborhood $U_0$ of $v_*$ such that if $v^0{\in}U_0$ then $v(t){\in}U_0$ for all $t{\ge}0$. See Definition~1 of Chapter~9 of~\citet{Hirsch} and Exercise~1 page~191, for example. \\
The condition $V^H_{N}(0){=}v^H{\in}\N^{|H|}$, for any $N{\ge}1$, for the convergence can clearly be  replaced by   $V^H_{N}(0){\in}K_0$, for any $N{\ge}1$, where $K_0$ is an arbitrary subset of $\N^{|H|}$.\\
We assume that $v_*{\in}U_0$ and we fix $s_0{\in}(0,t_0)$. Since, almost surely,
\[
\Lambda_\infty([s_0,t_0],\N^{|H|}){=}(t_0{-}s_0){>}0,
\]
there exists a finite subset $K_0$ of $\N^{|H|}$ and $\eta{>}0$ such that 
\[
\lim_{N\to+\infty} \P\left(\int_{s_0}^{t_0}\ind{V_N^{H}(s){\in}K_0}\diff s  > \eta\right)=1.
\]
Let
\[
\nu_N=\inf\{t>s_0: V^H_{N}(t){\in} K_0\},
\]
then the above relation gives that the sequence $(\P(\nu_N{<}t_0))$ converges to $1$. The state of $(V_N(t))$ at time $\nu_N$ satisfies the assumptions on the initial state for the convergence in distribution of the killed scaled process and the killed occupation measure. By the strong Markov property, by considering the process $(V_N(\nu_N{+}t))$, the convergence in distribution can be extended on the time interval $(0,2s_0)$. By induction, the convergence is then extended on $[0,{+}\infty)$. The proposition is proved. 
\end{proof}

\subsection{Networks of $M/M/\infty$ queues}\label{MMISec}
We introduce a classical birth and death process, the $M/M/\infty$ queue.  some technical results for this process play an important role in the proof of our convergence results.  See Chapter~6 of~\citet{Robert}. In Section~\ref{M1Sec}, the non-degenerate part of the fast processes of two cases is represented as a network of two $M/M/\infty$ queues. As it will be seen at the end of this section, its invariant distribution does not have a product form expression, nevertheless it has an explicit formulation. 
\begin{definition}\label{defMMI}
The ${M/M/\infty}$ queue with input rate $a{\ge}0$ and service rate $b{>}0$ is a  Markov process  on $\N$ with transition rates, for $x{\in}\N$,
\[
x\longrightarrow
\begin{cases}
x{+}1&   a,\\
x{-}1&   b x.
\end{cases}
\]
Its invariant distribution is $\Pois{a/b}$, a Poisson distribution with parameter $a/b$.
\end{definition}
An ${M/M/\infty}$ queue with input rate $a$ and service rate $b$ can be represented as
\begin{center}
\begin{figure}[h]
  \resizebox{4cm}{!}{%
    \begin{tikzpicture}[->,node distance=10mm]

\node(M) at (-2,0){$\emptyset$};
\node[black,rectangle,draw](L) at (0,0){${S}$};
\node(N) at (2,0){};

\path (M) edge [black,left,midway,above] node {$a$} (L);
\path (L) edge [black,left,midway,above] node {$b$} (N);
\end{tikzpicture}}
\caption{$M/M/\infty$ queue with input rate $a$ and service rate $b$}\label{FigF0}
\end{figure}
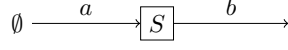
\end{center}
The state of an $M/M/\infty$ queue corresponds to the state of the  basic CRN
\[
\emptyset\mathrel{\mathop{\xrightleftharpoons[b]{a}}} {S}
\]
We now state some useful technical results on the $M/M/\infty$ queue. We begin with the following elementary lemma. 
\begin{lemma}\label{MMILower}
  If $(L(t))$ is the process of an ${M/M/\infty}$ queue with input rate $a{\ge}0$ and service rate $b{>}0$, if $L(0){=}N$, then there exists some $\eta{>}0$ and $t_0{>}0  $ such that
\begin{equation}\label{EqLemMMI}
\lim_{N\to+\infty} \P\left(\inf_{s\leq t_0}\frac{L(s)}{N}\ge \eta \right)=1.
\end{equation}
\end{lemma}
\begin{proof}
  If $(E_i,i{=}1,\ldots,N)$ is a vector of i.i.d. exponential random variables with parameter $b$, corresponding to the service times of $N$ initial customers, then
  \[
  \inf_{s\leq t_0}\frac{L(s)}{N}\ge \frac{1}{N}\sum_{i=1}^N\ind{E_i{\ge} t_0},
  \]
  and the lemma follows from the law of large numbers if $t_0$ and $\eta$ are chosen so that $\eta{<}\exp({-}bt_0)$. 
\end{proof}
The following classical proposition on the hitting times of the $M/M/\infty$ queue  is an important result that we will used repeatedly in the article. See Proposition~6.10 of~\citet{Robert}.
\begin{proposition}\label{MMIHit}
Let $(L(t))$ be the process of an $M/M/\infty$ queue with input rate $a{>}0$ and service rate $b{>}0$ with $L(0){=}\ell{\in}\N$, then for any  $\delta{>}0$, $T{>}0$,
    \[
    \lim_{N\to+\infty}\P\left(\sup_{t\leq T}\frac{L(N^\delta t)}{N}\ge 1\right)=0.
    \]
\end{proposition}
The following network of $M/M/\infty$ queues plays an important role in two cases of our study. See Section~\ref{M1Sec} and also in the case of~\ref{M1RSec} of the Appendix. It is associated to the fast processes of the averaging principle proved in these sections. 
\begin{center}
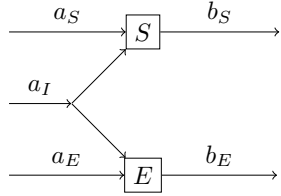
\begin{figure}[h]
\resizebox{4cm}{!}{%
\begin{tikzpicture}[->,node distance=20mm]
\node(L0) at (-2,0){};
\node[scale=0.04](L1) at (-1,0){};
\node(L2) at (-2,-1){};
\node(L3) at (-2,1){};
\node[black,rectangle,draw](X2) at (0,1){${ S}$};
\node[scale=0.75] (X2O) at (2,1){};
\node[black,rectangle,draw](UA) at (0,-1){${ E}$};
\node (UAO) at (2,-1){};
\path (L0) edge [black,left,above] node {$a_I$} (L1);
\path (L1) edge [black,midway,above] node {} (X2);
\path (L1) edge [black,midway,above] node {} (UA);
\path (L2) edge [black,left,above] node {$a_E$} (UA);
\path (L3) edge [black,left,above] node {$a_S$} (X2);
\path (X2) edge [black,left,above] node {$b_S$} (X2O);
\path (UA) edge [black,left,above] node {$b_E$} (UAO);
\end{tikzpicture}}
\caption{The $M/M/\infty$ Network}\label{FigF1}
\end{figure}
\end{center}
\begin{proposition}\label{MMN}[Network of $M/M/\infty$ queues]\label{MMNet}
For $C{=}(a_I,a_S,b_S,a_E,b_E){\in}(0,{+}\infty)^5$, the invariant probability distribution $\pi_C$ of the Markov process $(L(t)){=}(L_S(t),L_E(t))$ on $\N^2$ whose $Q$-matrix is given by, for $x{=}(x_S,x_E){\in}\N^2$, 
\begin{equation}\label{QMat}
x{\longrightarrow} x{+}
\begin{cases}
e_S{+}e_E,& a_I,\\
 e_S,&a_S,\\
 {-}e_S,&b_Sx_S,\\
\end{cases}
\qquad x{+}
\begin{cases}
e_E,& a_E,\\
{-}e_E,&b_Ex_E,\\
\end{cases}
\end{equation}
is such that
\begin{multline}\label{InvDistMMN}
  \croc{\pi_C,F}=\int_{\N^3}F(u+w,v+w)\\\Pois{\frac{a_S}{b_S}{+}\frac{a_I}{b_S}\frac{b_E}{(b_E{+}b_S)}}(\diff u)\Pois{\frac{a_E}{b_E}{+}\frac{a_I}{b_E}\frac{b_S}{(b_E{+}b_S)}}(\diff v)\Pois{\frac{a_I}{b_E{+}b_S}}(\diff w)
\end{multline}
for any $F{\in}{\cal C}_c(\N^2)$. In particular,
\begin{equation}\label{MMIExp}
  \int_{\N^2} x_S\pi_{C}(\diff x_S,\diff x_E)=\frac{a_I{+}a_S}{b_S}.
\end{equation}
\end{proposition}
The probability distribution $\pi_C$ does not has a product form but has nevertheless an explicit representation.
\begin{proof}
  Let $(L_I(t)){=}(L_1^I(t),L_2^I(t))$ be the Markov process starting from $(0,0)$, whose $Q$-matrix is given by Relation~\eqref{QMat} when $a_S{=}a_E{=}0$. Let $(L_S^+(t)$ and $(L_E^+(t))$ be two independent processes independent of $(L_I(t))$ and such that $(L_S^+(t))$, resp. $(L_E^+(t))$,  is an $M/M/\infty$ queue starting from $0$  with input rate $a_S$, resp.  $a_E$, and service rate $b_S$, resp. $b_E$. The invariant distribution of $(L_S^+(t))$, resp. $(L_E^+(t))$, is Poisson with parameter $a_S/b_S$, resp $a_E/b_E$. 

It is easily seen that, if $L(0){=}(0,0)$, then the process  $(L(t))$ has the same distribution as $(L_1^I(t){+}L_S^+(t),L_2^I(t){+}L_E^+(t))$,  by comparing the respective $Q$-matrices of these two Markov processes,. We are thus left to express the invariant distribution of $(L_I(t))$.

We will use the {\em coupling-from-the-past method}. The idea consists, via a convenient probabilistic representation of these stochastic processes,  in starting the process $(L(s))$ at time ${-}t$ and to study its state at time $0$. In this way, the distribution of the shifted process at time $0$ has the same distribution as the original process at time $t$. With an appropriate representation and provided that the process has convenient properties, like monotonicity for example, the state at time $0$ of the shifted process may converge {\em almost surely} as $t$ goes to inﬁnity.  In this case an explicit representation of the limit, and therefore of the limiting distribution,  is obtained.  See~\citet{Loynes} for one of the early use of this method, \citet{Peres} for a presentation of a variant, the {\em Propp-Wilson} algorithm,  and~\citet{Rob2019} for a use similar to the one done in this paper.   

Let ${\cal N}{=}(t_n,u_n,v_n)$ be a Poisson process on $\R{\times}\R_+^2$ with intensity measure
  \[
  a_I\diff t\otimes b_S e^{-b_Su}\diff u\otimes b_E e^{-b_Ev}\diff v.
  \]
The sequence $(t_n)$ is a Poisson process on $\R$ with rate  $a_I$ and $(u_n)$, resp. $(v_n)$,  is a sequence of i.i.d. exponential random variables with parameter $b_S$, resp. $b_E$. These three sequences are independent. See Proposition~1.11 of~\citet{Robert}.
Since $L(0)=(0,0)$ then, clearly, the relation
\begin{multline*}
  L_I(t)=\left(\sum_{n}\ind{0\le t_n\le t< t_n{+}u_n},\sum_{n}\ind{0\le t_n\le t< t_n{+}v_n}\right)\\
  \left(\int_{[0, t]\times \R_+^2}\ind{0\le s\le t< s{+}u} {\cal N}(\diff s,\diff u, \diff v),\int_{[0, t]\times \R_+^2}\ind{0\le s\le t< s{+}v} {\cal N}(\diff s,\diff u, \diff v)\right)
\end{multline*}
holds. For $t{>}0$ fixed, by  invariance of the distribution of the Poisson process ${\cal N}$ by the mappings $(s,u,v){\mapsto}(s{-}t,u,v)$ and $(s,u,v){\mapsto}({-}s,u,v)$, we have 
\begin{multline*}
  L_I(t)\steq{dist}  \left(\int_{[-t, 0]\times \R_+^2}\hspace{-4mm}\ind{ s\le 0< s{+}u} {\cal N}(\diff s,\diff u, \diff v),\int_{[-t, 0]\times \R_+^2}\hspace{-4mm}\ind{s\le 0< s{+}v} {\cal N}(\diff s,\diff u, \diff v)\right)\\
  \stackrel{t\to+\infty}{\longrightarrow} 
  \left(\int_{(-\infty, 0]\times \R_+^2}\hspace{-4mm}\ind{ s\le 0< s{+}u} {\cal N}(\diff s,\diff u, \diff v),\int_{(-\infty, 0]\times \R_+^2}\hspace{-4mm}\ind{s\le 0< s{+}v} {\cal N}(\diff s,\diff u, \diff v)\right)\\
  \steq{dist}
    L_I^\infty \steq{def}\left(\int_{[0, +\infty)\times \R_+^2}\hspace{-4mm}\ind{  0\le s< u} {\cal N}(\diff s,\diff u, \diff v),\int_{[0, +\infty)\times \R_+^2}\hspace{-4mm}\ind{0\le s< v} {\cal N}(\diff s,\diff u, \diff v)\right).
\end{multline*}
See Proposition~1.12 of~\citet{Robert}. The law of the random variables of the right-hand side of the last expression is therefore the invariant distribution of $(L_I(t))$. The proof is concluded by calculating the Laplace transform of $L_I^\infty$, by using the formula for the Laplace transform of the Poisson process ${\cal N}$: If $f$ is a non-negative Borelian function on $\R{\times}\R_+^2$, then
\begin{multline*}
  \E\left(\exp\left(-\int_{\R{\times}\R_+^2}f(s,u,v){\cal N}(\diff s,\diff u,\diff v)\right)\right)\\
  =\exp\left(-\int_{\R{\times}\R_+^2}\left(1{-}e^{-f(s,u,v)}\right)  a_I\diff s b_S e^{-b_Su}\diff u b_E e^{-b_Ev}\diff v\right).
\end{multline*}
See Proposition~1.5 of~\citet{Robert}.
\end{proof}
\section{Underloaded Case, $e{+}f{<}1$}\label{M1Sec}
In this section we consider three cases for the subset $H$
\begin{enumerate}
\item $H{=}\{{ S}_2,{ S}_3, { A},  { AS_2}, { BS_3}\}$;
\item $H{=}\{{ A}, { B}\}$;
\item $H{=}\{{ S}_1, { S}_2,  { AS_1},{ B},  { BS_2}\}$;
\end{enumerate}
For each of these regimes, a convergence theorem on a convenient finite time interval $[0,t_0)$ is established as well as a stability property of an equilibrium of the limiting dynamical system. The stability of $H$ is a consequence of Proposition~\ref{StabProp}. Only a proof of the theorem of case~(1) is carried out. The other theorems are proved with similar arguments. 

Section~\ref{M1LSec} is for case~(1), Section~\ref{M1MSec} is for case~(2) and, since it is very similar to case~(1),  Section~\ref{M1RSec} of the appendix is just a quick formulation for the case~(3).

\subsection{Weight on the Left}\label{M1LSec}
\addcontentsline{toc}{section}{ \ \thesubsection\quad Weight on the Left}
We will show  that the set $H{=}\{{ S}_2,{ S}_3, { A},  { AS_2}, { BS_3}\}$ is stable under the condition $\lambda_1e{<}\mu_2 f$. 

The following theorem establishes a convergence in distribution  result on a finite time interval  for the set $H$.
\begin{theorem}\label{ThM1L}
Under the condition $e{+}f{<}1$, if $H{=}\{{ S}_2,{ S}_3, { A},  { AS_2}, { BS_3}\}$ and the initial condition is 
\[
V_N^H(0)=(x_2,x_3,u_A,a_2,b_1)\in\N^5, \quad V_N^{H^c}(0)=(x_1^N,a_1^N,u_B^N,b_2^N),
\]
such that 
\[
\lim_{N\to+\infty}\frac{V_N^{H^c}(0)}{N}=(x_1^0,a_1^0,u_B^0,b_2^0)\steq{def}(1{-}e{-}b_2^0,e,f{-}b_2^0,b_2^0),
\]
for some $b_2^0{\in}(0,f)$, then there exists $t_0{>}0$ such that, for the convergence in distribution
\begin{multline}\label{CVWL}
\lim_{N\to+\infty}\left(\left(\frac{V_N^{H^c}(t)}{N},t{<}t_0\right),\Lambda^{H,t_0}_N\right) \\=\left((v(t),t{<}t_0),\Lambda^{t_0}_{\infty}\right)=
\left(\left((1{-}e{-}{b}_2(t),e,f{-}{b}_2(t),{b}_2(t)), t{<}t_0\right),\Lambda^{t_0}_\infty\right),
\end{multline}
where
\[
{b}_2(t)=b_2^0\exp(-\mu_2 t)+\frac{\lambda_1 e}{\mu_2}\left(1{-}\exp(-\mu_2 t)\right),
\]
 and $\Lambda_\infty$ is the random measure defined by, for $F{\in}{\cal C}_c(\R_+{\times}\N^5)$, 
\begin{multline*}
\int_{\R_+{\times}\N^5} F(s,\underbrace{(x_2,x_3,u_A,a_2,b_1)}_{v^H})\Lambda_{\infty}(\diff s, \diff v^H)\\ =\int_{\R_+{\times}\N^3} F(s,(x_2,0,u_A,L_{1}(s),L_{2}(s)))\pi_{C(s)}(\diff x_2,\diff u_A)\diff s
\end{multline*}
holds almost surely, where, for $C{\in}(0,{+}\infty)^5$,  $\pi_{C}$ is the invariant measure of Proposition~\ref{MMNet}, and, for $t{\ge}0$, 
\begin{equation}\label{eqCt}
C(t){=}\left(\rule{0mm}{5mm}\lambda_1e,\beta_2^-{b}_2(t),\beta_2(f{-}{b}_2(t)),\alpha_1^-e,\alpha_1(1{-}e{-}{b}_2(t))\right).
\end{equation}
The stochastic process $(L_1(t)),L_2(t))$ is associated to a {\em non-homogeneous network} of $M/M/\infty$ queues, with initial state $(a_2,b_1)$ and in state $\ell{=}(\ell_1,\ell_2)$ at time $t$, its transition rates are 
\begin{equation}\label{SlowTrans}
    \ell{\longrightarrow}   \ell {+} \begin{cases}
\scriptstyle{e_{1}}, &\scriptstyle{\alpha_2r_L(t)},\\
\scriptstyle{{-}e_{1}}, &\scriptstyle{\alpha_2^-\ell_1},\\
\scriptstyle{{-}e_{2}}, &\scriptstyle{\mu_1\ell_2},\\
\scriptstyle{e_{2}{-}e_{1}}, &\scriptstyle{\lambda_2\ell_1},
   \end{cases}
\qquad   \text{ with }   r_L(t)=\int_{\N^2} x_2u_A\pi_{C(t)}(\diff x_2,\diff u_A).
\end{equation}
\end{theorem}
The expression of the limit $\Lambda_\infty$ involves a random component, the process $(L_1(t),L_2(t))$.
\begin{proof}

We fix some $\eps{>}0$, which will be determined later, and define $\tau_N$ as the first instant when one of the coordinates of $(V_N^{H^c}(t))$,
is below $\eps N$: 
  \[
  \tau_N\steq{def}\inf\{t{>}0: \min(X_1^N(t),A_1^N(t),U_B^N(t),B_2^N(t))<\eps N\}.
  \]
The strategy of the proof is simple. It consists in the construction of couplings up to time $\tau_N$ of several sets of coordinates with specific, easier to study, processes. 
One of the difficulties is of taking care of the various dependencies implied by these couplings. Recall that there are {\em twelve} Poisson processes driving the time evolution of $(V_N(t))$. We will also use stochastic calculus arguments, for this reason these couplings have to be constructed explicitly in the current probability space.

These couplings will give that there exists some $t_0{>}0$ such that the relation $\tau_N{\ge}t_0$ holds with high probability when $N$ gets large, and that the sequence of  random variables $((\overline{V}_N^{H^c}(t{\wedge}t_0)),(\Lambda_N^{H,t_0}))$ is tight for the convergence in distribution. There is an averaging principle to be proved for which the fast variables are  $(U_A^N(t),X_2^N(t))$.
 The ``slow'' and $O(1)$ variables $(A_2^N(t),X_3^N(t),B_1^N(t))$ are studied separately.
The proof is concluded with the help of Proposition~\ref{MMNet}  and standard arguments to identify the possible limits.

To ease the reading of the proof (hopefully), we have divided it as a set of eight subsections on specific issues. The last part of the theorem on the process  $(L_1(t),L_2(t))$ as an asymptotic description of the slow $O(1)$ processes $(A_2^N(t),B_1^N(t))$ is proved in Section~\ref{O1SSec} of the appendix.

\subsubsection{The process $(U_A^N(t))$}\label{CoupUASec}
Two couplings are achieved for this process. The first one is carried out in full details, and is used to illustrate the future couplings arguments of the article. It is used in the next Section to study the growth of the process $(X_1^N(t))$. See Section~\ref{CoupX1Sec}. The second one controls controls the sequence of two-dimensional processes $(X_2^N(t), U_A^N(t))$, and is carried out with less details. See Section~\ref{CoupX2Sec}.

Let $\delta_0=\max(\alpha_1^-,\lambda_1,\alpha_2^-,\lambda_2)$ and $(W_A^N(t))$ the solution of the SDE
\begin{multline}\label{WA}
\diff W_A^N(t)=\cal{P}_{\lambda_1}((0, \delta_0 A_1^N(t{-})), \diff t) 
{+}\cal{P}_{\alpha_1^-}((0, \delta_0 A_1^N(t{-})), \diff t) \\
{+}\cal{P}_{\lambda_2}((0, \delta_0 A_2^N(t{-})), \diff t)
{+}\cal{P}_{\alpha_2^-}((0, \delta_0 A_2^N(t{-})), \diff t)\\
{+}\cal{P}_{\delta_0}((0, 2\delta_0 (E_N{-}A_1^N(t{-}){-}A_2^N(t{-}))), \diff t)
{-}\cal{P}_{\alpha_1}((0,\eps\alpha_1NW_A^N(t{-})), \diff t),
\end{multline}
with $W_A^N(0){=}u_A$, where $\cal{P}_{\delta_0}$ is a Poisson process independent of the Poisson processes $({\cal P}_\kappa, \kappa{\in}{\cal R}_P)$, and with the same distribution. 

It is easily seen that $(W_A^N(t)){=}(Q_{0}(Nt))$, where $(Q_0(t))$ is the process of an $M/M/\infty$ queue with input rate $2\delta_0$ and service rate $\eps\alpha_1$.  By induction on the instants of successive jumps of the process $(V_N(t),W_A^N(t))$, we get that the relation $U_A^N(t){\le}W_A^N(t)$ holds for all $t{<}\tau_N$.

\subsubsection{The  process $(X_1^N(t))$}\label{CoupX1Sec}
For $t{<}\tau_N$, the coupling of Section~\ref{CoupUASec} and the SDEs~\eqref{SDE} give the relation
\[
  X_1^N(t)\ge x_1^N-\int_{(0,t]}{\cal P}_{\alpha_1}((0,\eps \alpha_1NU_A^N(s)),\diff s)\ge  Y_1^N(t)
      \]
      with
\[
 Y_1^N(t)\steq{def} x_1^N{-}\int_{(0,t)}\cal{P}_{\alpha_1}((0,\eps\alpha_1NW_A^N(s{-})), \diff s).
\]
Since $(W_A^N(t))$ is an adapted \cadlag process, the process 
\[
 \left(\frac{1}{N}\int_{(0,t)}\left(\cal{P}_{\alpha_1}((0,\eps\alpha_1NW_A^N(s{-})), \diff t)-\eps\alpha_1 N W_A^N(s)\diff s\right)\right)
\]
is a martingale whose previsible increasing process is
\[
\left(\frac{\eps\alpha_1}{N^2}\int_0^{Nt} Q_0(s)\diff s\right).
\]
The ergodic theorem for the positive recurrent Markov process $(Q_0(t))$, with Doob's Inequality, gives  the convergence in distribution,
\begin{equation}\label{Coupeq1}
\lim_{N\to+\infty}  \left(\frac{Y_1^N(t)}{N}\right)=\left(x_1^0-2\delta_0t\right).
\end{equation}

\subsubsection{The process $(X_2^N(t))$}\label{CoupX2Sec}
We now proceed with the control of the second coordinate in the same way as for $(U_A^N(t))$. Because of the dependence of some of the jumps of $(U_A^N(t))$ and $(X_2^N(t))$, via the reaction with rate $\lambda_1$ for example, we will achieve it by a coupling of both coordinates. The coupling construction follows the same method as for $(U_A^N(t))$. Let $(Q_1^N(t),Q_2^N(t))$ be a Markov process on $\N^2$ starting from $(u_A,x_2)$ with the following set of transitions and rates
\[
(q_1,q_2)\longrightarrow
\begin{cases}
  (q_1+1,q_2)&(\alpha_1^-{+}\alpha_2^-{+}\lambda_2)N\\
  (q_1-1,q_2)&\eps\alpha_1N q_1\\
  (q_1+1,q_2+1)&\lambda_1 N
\end{cases}
\qquad
\begin{cases}
  (q_1,q_2+1)&(\mu_1+\beta_2^-)N\\
  (q_1,q_2-1)&\eps\beta_2N q_2.
\end{cases}
\]
By checking the possible jumps of both processes, one can construct a coupling  of $(Q_1^N(t),Q_2^N(t))$ such that
$U_A^N(t){\le} Q_1^N(t)$ and $X_2^N(t){\le}Q_2^N(t)$ for all $t{\le}\tau_N$. Note that
$(Q_1^N(t),Q_2^N(t)){=}(Q_1(Nt),Q_2(Nt)))$ where $(Q_1(t),Q_2(t))$ is the Markov process associated to the network of $M/M/\infty$ queues of Proposition~\ref{MMNet} for the parameter $C_Q{=}(\lambda_1,\alpha_1^-{+}\alpha_2^-{+}\lambda_2,\alpha_1\eps,\mu_1+\beta_2^-,\beta_2\eps)$.

\subsubsection{The other $O(1)$ processes}\label{CoupO1Sec}
In view of Relations~\eqref{SDE}, for $t{>}0$,  the random variables $A_2^N(t)$, $X_3^N(t)$ and  $B_1^N(t)$ can be upper-bounded by $s_0{=}a_2{+}x_3{+}b_1$, plus the number of jumps of the chemical reaction $A{+}S_2{\rightharpoonup}AS_2$ between $0$ and $t$. Hence, on the event $\{t{<}\tau_N\}$,  the relation
\begin{multline}\label{CoupIneq1}
\max(A_2^N(t),X_3^N(t),B_1^N(t))\\\leq s_0{+}\int_{(0,t]}\cal{P}_{\alpha_2}((0,U_A^N(s{-})X_2^N(s{-})), \diff s)\le  s_0{+}S^0_N(t),
\end{multline}
with
\[
    S^0_N(t)\steq{def} \int_{(0,t]}\hspace{-4mm}\cal{P}_{\alpha_2}((0,\alpha_2Q_1(Ns{-})Q_2(Ns{-})), \diff s).
\]
holds. 

The Markov process  $(Q_1(t),Q_2(t))$ being positive recurrent with invariant distribution $\pi_{C_Q}$, see Proposition~\ref{MMNet}, it is easily checked that
\[
\int_{\N^2}u_Ax_2\pi_{C_Q}(\diff u_A,\diff x_2) < {+}\infty.
\]
Since,
\[
\E\left(S_N^0(t)\right)=\alpha_2\E\left(\frac{1}{N}\int_0^{Nt}Q_1(s)Q_2(s)\diff s\right),
\]
the ergodic theorem for Markov processes gives the relation, for $T{>}0$, 
\[
\limsup_{N\to+\infty} \sup_{t\le T} \E\left(S_N^0(t)\right)\le\alpha_2T\int_{\N^2}u_Ax_2\pi_{C_Q}(\diff u_A,\diff x_2).
\]
By gathering the above estimations, we  have thus obtained that the constant
\begin{equation}\label{CoupO1Eq}
C_0(T)\steq{def}\sup_{N}\sup_{t\leq T} \E\left(\left(U_A^N(t){+}X_2^N(t){+}A_2^N(t){+}X_3^N(t){+}B_1^N(t)\right) \ind{T<\tau_N}\right) 
\end{equation}
is finite. 
\subsubsection{The processes $(A_1^N(t))$ and $(B_2^N(t))$}\label{CoupA1Sec}
It is easily seen that there is a coupling of these two processes with two $M/M/\infty$ queues $(Q_3(t))$ and $(Q_4(t))$ with input rate $0$ and with respective service rates  $\lambda_1{+}\alpha_1^-$ and $\mu_2{+}\beta_2^-$, such that, $A_1^N(t){\ge}Q_3(t)$ and $A_1^N(t)\ge Q_4(t)$ holds for all $t{\ge}0$. 
\subsubsection{The process $(U_B^N(t))$}
For $t{<}\tau_N$, the SDEs~\eqref{SDE} give the relation, for $t{<}\tau_N$,
\begin{multline}\label{EqCoupUB}
U_B^N(t) \ge u_B^N-\int_{(0,t]}{\cal P}_{\beta_1}((0, \beta_1U_B^N(s{-})  X_3^N(s{-})),\diff s)\\
-\int_{(0,t]}{\cal P}_{\beta_2}\left(\left(0,\beta_2U_B^N(s{-}) X_2^N(s{-})\right),\diff s\right).
\end{multline}
The first integral $(I_1^N(t))$ of the right-hand side of the last relation is the number of jumps due the Poisson process ${\cal P}_{\beta_1}$ between $0$ and $t$. With the same argument as in Section~\ref{CoupO1Sec}, we have
\[
I_1^N(t)\steq{def}\int_{(0,t]}{\cal P}_{\beta_1}((0, \beta_1U_B^N(s{-})  X_3^N(s{-})),\diff s)\le s_0{+}  S_N^0(t),
  \]
  where $(S_N^0(t))$  is defined by Relation~\eqref{CoupIneq1}.
  It is not difficult to prove that $(S_N^0(t)/N)$ converges in distribution to $0$.

  For the second integral $(I_2^N(t))$,
\begin{multline*}
  I_2^N(t)\steq{def}\int_{(0,t]}{\cal P}_{\beta_2}\left(\left(0,\beta_2U_B^N(s{-}) X_2^N(s{-})\right),\diff s\right)\\ 
\le  J_2^N(t)\steq{def} \int_{(0,t]}{\cal P}_{\beta_2}\left(\left(0,\beta_1E_N Q_2(Ns{-})\right),\diff s\right).
\end{multline*}
   With the same arguments as in Section~\ref{CoupX1Sec}, there exists some constant $C_U$ such that, for the convergence in distribution,
   \[
   \lim_{N\to+\infty} \left(\frac{J_2^N(t)}{N}\right)=(C_Ut). 
   \]
   Now we gather the above estimations based on couplings valid up to time $t$. On the time interval $[0,\tau_N]$ we have stochastic lower bounds of
   \begin{itemize}
   \item $(X_1^N(t))$ by $(Y_1^N(t))$ and Relation~\eqref{Coupeq1};
   \item $(A_1^N(t))$ and $(B_2^N(t))$ in Section~\ref{CoupA1Sec} by using Lemma~\ref{MMILower}, and Relation~\eqref{EqCoupUB} for $(U_B^N(t))$.
   \end{itemize}
We can therefore find $\eps{>}0$ and $t_0>0$ depending only on $(x_1^0,a_1^0,u_B^0,b_2^0)$ such that all these stochastic lower bounds are all greater than $\eps$ on the time interval $[0,t_0]$ with high probability as $N$ gets large. For example for $(X_1^N(t))$, due to Relation~\eqref{Coupeq1}, the condition $t_0{\in}(0,(x_1^0{-}\eps)/4\delta_0)$ is sufficient. We obtain that, for a convenient $t_0{>}0$,  the sequence $(\P(\tau_N{\ge}t_0))$ is converging to $1$ as $N$ gets large.
\subsubsection{Tightness Properties}\label{CoupFast}
   We first prove that the sequence $(\Lambda_N^{H,t_0})$ is tight. Let, for $K{>}0$,  ${\cal K}{=}[0,K]^5$. Then, for any $t{<}t_0$,
\begin{multline*}
   \E\left(\Lambda_N^{H,t_0}\left([0,t]{\times}{\cal K}^c\right)\right)
   \le t_0\P(\tau_N<t_0)\\+\int_0^{t}\P\left(U_A^N(s){+}X_2^N(s){+}A_2^N(s){+}X_3^N(s){+}B_1^N(s)\ge \frac{K}{5},\tau_N>t_0\right)\diff s
\\   \le t_0\P(\tau_N<t_0)+\frac{5C_0(t_0)}{K}t,
\end{multline*}
by Relation~\eqref{CoupO1Eq} in Section~\ref{CoupO1Sec}. We conclude that $(\Lambda_N^{H,t_0})$ is tight by using Lemma~1.3 of~\citet{Kurtz1992}. Lemma~1.4 of this reference gives that any limiting point $\Lambda_{\infty,t_0}$ of $(\Lambda_N^{H,t_0})$ can be represented as, for $F{\in}{\cal C}_c(\R_+{\times}\N^5)$, 
\begin{equation}\label{GamEq}
\croc{\Lambda_{\infty,t_0},F}=\int_{\R_+{\times}\N^5} F(s,v)\gamma_s(\diff v)\diff s,
\end{equation}
almost surely, where $(\gamma_s)$ is an optional process with values in ${\cal M}_P(\N^5)$. 
To simplify the expressions (to avoid subsequences), we assume that $(\Lambda_N^{H,t_0})$ is converging in distribution to such a random measure $\Lambda_{\infty,t_0}$. 

The SDEs~\eqref{SDE} and straightforward stochastic calculus give, with the notations of Relation~\eqref{ScalNot},  the relation
\begin{equation}\label{B2Eq}
\overline{B}_2^N(t)=\frac{b_2^N}{N}+M_N(t){+}\beta_2\int_0^t\overline{U}_B^N(s)X_2^N(s)\diff s-(\mu_2{+}\beta_2^-)\int_0^t\overline{B}_2^N(s)\diff s,
\end{equation}
for $t\ge 0$, where $(M_N(t))$ is a martingale whose previsible increasing process is given by
\[
\left({M}_N(t)\right)=\left(\frac{1}{N}\left(\beta_2\int_0^t\overline{U}_B^N(s)X_2^N(s)\diff s{+}(\mu_2{+}\beta_2^-)\int_0^t\overline{B}_2^N(s)\diff s\right)\right). 
\]
With Relation~\eqref{eqCons2}, we have $\overline{U}_B^N(t){+}\overline{B}_2^N(t){\le}F_N/N$, the inequality $X_2^N(t){\le} Q_2(Nt)$ is valid for any $t{\le}\tau_N$ of Section~\ref{CoupX2Sec},  and Doob's Inequality give that the martingale $(M_N(t),t{<}t_0)$  is converging in distribution to $0$. With the criterion of the modulus of continuity, see Theorem~7.3 of~\citet{Billingsley}, we get that the sequence of processes $(\overline{B}_2^N(t),t{<}t_0)$ is tight for the convergence in distribution and any of its limiting point  $(b_2(t))$ is a continuous process.

As before, we assume that the sequence of processes $(\overline{B}_2^N(t),t{<}t_0)$ is converging in distribution to some process $(b_2(t))$. 
The relations  $U_A^N(t){\le} Q_1(Nt)$ and $X_2^N(t){\le}Q_2(Nt)$ for all $t{\le}\tau_N$ of Section~\ref{CoupX2Sec}, with Proposition~\ref{MMIHit} and Relation~\eqref{CoupIneq1} give the convergence in distribution
\begin{equation}\label{CoupO1Ineq1}
\lim_{N\to+\infty}\left(\frac{X_2^N(t)}{N},\frac{X_3^N(t)}{N},\frac{U_A^N(t)}{N},\frac{A_2^N(t)}{N},\frac{B_1^N(t)}{N}, 0\le t<t_0\right)=(0).
\end{equation}
The relations of  mass conservation~\eqref{eqCons1} and~\eqref{eqCons2} give therefore the tightness property of $(\overline{V}_N^{H^c}(t))$, and  the convergence
\begin{multline}\label{CvM1LT}
\lim_{N\to+\infty}\left(\left(\overline{X}_1^N(t),\overline{A}_1^N(t),\overline{U}_B^N(t),\overline{B}_2^N(t)\right),t{<}t_0\right)\\=
((1-e-b_2(t),e,f-b_2(t),b_2(t)), t{<}t_0),
\end{multline}
when $(b_2(t))$ is the limiting point of $(\overline{B}_2^N(t))$.

\subsubsection{Fast $O(1)$ Variables}\label{O1FSec}
The $O(1)$ variables $(X_2^N(t), U_A^N(t))$ are ``fast''. For simplicity, we give their transition rates at $t=0$,
\begin{equation}\label{O1SHEq1}
v{\longrightarrow}   v{+} 
\begin{cases}
\scriptstyle{e_{U_A}{+}e_{X_2}}, &\scriptstyle{\lambda_1a_1^N},\\
\scriptstyle{e_{U_A}}, & \scriptstyle{\alpha_1^-a_1^N},\\
\scriptstyle{{-}e_{U_A}}, &\scriptstyle{\alpha_1 u_Ax_1^N},
\end{cases}
\quad
\begin{cases}
\scriptstyle{{-}e_{X_2}}, &\scriptstyle{\beta_2x_2u_B^N},\\
\scriptstyle{e_{X_2}}, &\scriptstyle{\beta_2^-b_2^N}.
\end{cases}
\end{equation}
At time $t$, $a_1^N$ is replaced by $A_1^N(t)$ and similarly for the other slow variables. 

The occupation measure associated to $(X_2^N(t), U_A^N(t))$ is defined as $\Xi_N$ such that, for $G{\in}{\cal C}_c([0,t_0){\times}\N^2)$,
\begin{equation}\label{XiN}
\croc{\Xi_N,G}=\int_0^{t_0} G(s,X_2^N(s), U_A^N(s))\diff s. 
\end{equation}
Provided that Relation~\eqref{CvM1LT} holds, following the standard method of~\citet{Kurtz1992}, we get that
$(\Xi_N)$ is converging in distribution to $\Xi_\infty$ defined by
\[
\croc{\Xi_\infty,G}=\int_0^{t_0} G(s,x_2, u_A)\pi_{C(s)}(\diff x_2, \diff u_A)\diff s,
\]
almost surely, $(C(t))$ is given by relation~\eqref{eqCt}. We give a sketch of the method: If $f$ is a function on $\N^2$ with finite support, with the SDEs~\eqref{SDE}, we  write the It\^o's formula for  $f(X_2^N(t), U_A^N(t))$, this relation is divided by  $N$ and, when $N$  go to infinity, this gives a characterization of the limit $\Xi_\infty$ in terms of the invariant distribution of the Markov process~\eqref{O1SHEq1}. See the proofs of Proposition~14 of~\cite{LR24-2}, or of Theorem~5 of~\cite{LR24} for an illustration of this method. 

This gives a characterization of the image of the process $(\gamma_t)$ defined by Relation~\eqref{GamEq}  by the projection on the coordinates $(x_2, u_A)$ as the continuous process $(\pi_{C(t)})$. 

By letting $N$ go to infinity in Relation~\eqref{B2Eq} and using the convergence in distribution of $(\Xi_N)$, Relation~\eqref{MMIExp} and the identity~\eqref{CvM1LT},
\[
b_2(t)=b_2^0+\lambda_1e t-\mu_2\int_0^t b_2(s)\diff s
\]
holds. This implies the convergence in distribution of $(\overline{B}_2^N(t))$  on the time interval $[0,t_0)$ and therefore of the process $(\overline{V}_N^{H^c}(t))$ by Relation~\eqref{CvM1LT}.
  
To conclude the proof of the theorem, the asymptotic behavior of the random variables  $(X_3^N(t),A_2^N(t),B_1^N(t))$ is investigated in Section~\ref{O1SSec} of the Appendix. 
\end{proof}
\begin{corollary}
Under the conditions of Theorem~\ref{ThM1L} if $\lambda_1e{<}\mu_2 f$ then the subset $H{=}\{{ S}_2,{ S}_3, { A},  { AS_2}, { BS_3}\}$ is stable for the sequence of Markov processes $(V_N(t))$. 
\end{corollary}
\begin{proof}
The proof is straightforward, using Proposition~\ref{StabProp}. Because of the limit of the process $(\overline{U}_B^N(t))$, the constant $t_0$ of the above theorem depends of the fact that $b_2^0{\in}(0,f)$. The function $(b_2(t))$ is converging to $\lambda_1 e/\mu_2{\in}(0,f)$ as $t$ goes to infinity, therefore $t_0$ can be arbitrarily large. 
\end{proof}
\subsection{Few Free Enzymes}\label{M1MSec}
\addcontentsline{toc}{section}{ \ \thesubsection\quad Few Free Enzymes}

In Sections~\ref{M1LSec} and~\ref{M1RSec} in the appendix, the regimes investigated involve one species of enzymes which is fully utilized while free enzymes of the other species are $O(N)$. For the regime investigated in this section, the situation is more balanced: Most of enzymes ${ A}$ and ${ B}$ are bound with substrat.

The analogue of Theorems~\ref{ThM1L} and~\ref{ThM1R} in the appendix  is the following convergence result. The ingredients of its proof are similar  to the ones in the proof of Theorem~\ref{ThM1L}. Note that there is no equivalent of the processes $(L_1(t),L_2(t))$ in the asymptotic  description of the occupation measure. 

\begin{theorem}\label{ThM1M}
Under the condition $e{+}f{<}1$, if $H{=}\{{ A}, { B}\}$ and $V_N(0){=}v_N$,  with
\[
V_N^H(0)=(u_A,u_B)\in\N^2, \quad V_N^{H^c}(0)=(x_1^N,x^N_2,x^N_3, a_1^N,a_2^N,b_1^N,b_2^N)\in\N^7,
\]
such that
\[
\lim_{N\to+\infty}\frac{V_N^{H^c}(0)}{N}=v_0^{H^c}=(x_1^0,1{-}e{-}f{-}x_1^0{-}x_3^0,x_3^0,a_1^0,e{-}a_1^0,b_1^0,f{-}b_1^0),
\]
with $(x_1^0,x_3^0,a_1^0,b_1^0){\in}{\cal E}_P$,
\[
{\cal E}_P\steq{def}\{(x_1,x_3,a_1,b_1){\in}(0,1)^2{\times}(0,e){\times}(0,f): x_1{+}x_3<1{-}e{-}f\},
\]
 then there exists $t_0{>}0$ such that, for the convergence in distribution
\begin{multline}\label{CVWM}
  \lim_{N\to+\infty}\left(\left(\frac{V_N^{H^c}(t)}{N}, t{<}t_0\right),\Lambda_N^{H,t_0}\right)=((v(t),t{<}t_0),\Lambda_\infty^{t_0})\\
\steq{def}\left(\left(\left({x}_1(t),{x}_2(t),{x}_3(t),{a}_1(t),{a}_2(t),{b}_1(t),{b}_2(t)\right), t{<}t_0\right),\Lambda_\infty^{t_0}\right)
\end{multline}
where
\begin{enumerate}
  \item  $({x}_1(t), {x}_3(t), {a}_1(t), {b}_2(t))$ is the solution of the ODE 
\begin{equation}\label{ODEM}
\begin{cases}
\dot{{x}}_1(t)&=\mu_2{b}_2(t){+}\alpha_1^-{a}_1(t){-}\alpha_1{x}_1(t) m_A(t),\\ 
\dot{{x}}_3(t)&= \lambda_2 {a}_2(t){+}\beta_1^-{b}_1(t){-}\beta_1{x}_3(t)m_B(t),\\
\dot{{a}}_1(t)&= \alpha_1m_A(t){x}_1(t){-}(\lambda_1{+}\alpha_1^-){a}_1(t),\\
\dot{{b}}_1(t)&=\beta_1m_B(t){x}_3(t)){-}(\mu_1{+}\beta_1^-){b}_1(t),
\end{cases}
\end{equation}
with  initial state $({x}_1^0, {x}_3^0, {a}_1^0, {b}_2^0)$,  and, for $t{\ge}0$,
\[
m_A(t)\steq{def}\frac{(\lambda_1{+}\alpha_1^-) {a}_1(t)+  (\lambda_2 {+}\alpha_2^-){a}_2(t)}{\alpha_1 {x}_1(t)+  \alpha_2 {x}_2(t)},\quad
m_B(t)\steq{def}\frac{(\mu_1{+}\beta_1^-) {b}_1(t)+  (\mu_2{+}\beta_2^-) {b}_2(t)}{\beta_1 {x}_3(t)+  \beta_2 {x}_2(t)}.
\]
and  $x_2(t){=}1{-}e{-}f{-}{x}_1(t){-}{x}_3(t)$, $a_2(t){=}e{-}{a}_1(t)$, $b_2(t){=}f{-}{b}_1(t)$.\\
\item The measure $\Lambda_\infty$ is defined by, for $F{\in}{\cal C}_c(\R_+{\times}\N^2)$, 
\[
  \croc{\Lambda_\infty,F}=\int_{\R_+{\times}\N^2} F(s,u,v)  \Pois{m_A(s)}(\diff u)\Pois{m_B(s)}(\diff v)\diff s,
  \]
  almost surely.
\end{enumerate}

\end{theorem}
\begin{proof}
The proof is analogous, even simpler, to the proof of Theorem~\ref{ThM1L}. It is therefore skipped. 
It should be noted that the $O(1)$ processes are all ``fast''  in this case, contrary to Theorem~\ref{ThM1L} and Theorem~\ref{ThM1R} of the appendix. In this case their transitions are, at time $t=0$ and starting from $V_N(0)$,
\[
v^H\longrightarrow v^H{+}
\begin{cases}
e_{U_A}&(\alpha_1^-{+}\lambda_1)a_1^N+(\alpha_2^-{+}\lambda_2)a_2^N,\\
{-}e_{U_A}&(\alpha_1x_1^N{+}\alpha_2x_2^N)u_A,\\
e_{U_B}&(\beta_1^-{+}\mu_1)a_1^N+(\beta_2^-{+}\mu_2)a_2^N,\\
{-}e_{U_B}&(\beta_1x_3^N{+}\beta_2x_2^N)u_B.
\end{cases}
\]
and then evolve slowly, keeping the same order of magnitude, see comments of Relation~\eqref{O1SHEq1}. 
\end{proof}

\begin{proposition}[Equilibrium]\label{EquiM2M}
Under the conditions
\begin{equation}\label{CondM}
e+f < 1 \text{ and }  \left(\mu_1f-\lambda_2e\right)\left(\lambda_1e-\mu_2f\right) >0 
\end{equation}
  there exists a unique equilibrium $v_*{\steq{def}}(x_1^*,x_3^*,a_1^*,b_1^*)$ of the dynamical system~\eqref{ODEM} given by
\begin{equation}\label{EquiM1M}
  \begin{cases}
    a_1^*& \displaystyle =\mu_2\frac{\mu_1f-\lambda_2 e}{\lambda_1\mu_1-\lambda_2\mu_2},\quad b_1^*=\lambda_2\frac{\lambda_1e-\mu_2f}{\lambda_1\mu_1-\lambda_2\mu_2},\medskip\\
  x_1^*  & \displaystyle=\frac{1}{D}\alpha_2  \beta_1  \lambda_1  \mu_2  \left(1{-}e {-}f\right) \left(\mu_1 f {-}\lambda_2e \right)^{2} \left(\mu_2  {+}\beta_2^{-} \right) \left(\lambda_1  {+}\alpha_1^{-} \right),\medskip\\
x_3^* & \displaystyle=\frac{1}{D}\alpha_1  \beta_2  \lambda_2 \mu_1 \left(1{-}e{-}f\right) \left( \lambda_1e  {-} \mu_2f  \right)^{2}   \left(\mu_1  {+}\beta_1^{-} \right) \left(\lambda_2  {+}\alpha_2^{-} \right),
  \end{cases}
\end{equation}
with
\begin{multline*}
  D=\alpha_2  \beta_1  \lambda_1  \mu_2  \left(\lambda_2e {-}\mu_1 f  \right)^{2} \left(\mu_2  {+}\beta_2^{-} \right) \left(\lambda_1  {+}\alpha_1^{-} \right)\\
  {+}\alpha_1  \beta_1  \lambda_1  \mu_1  \left(  \mu_1 f {-}\lambda_2e   \right) \left(\lambda_1e {-} \mu_2f  \right) \left(\mu_2  {+}\beta_2^{-} \right) \left(\lambda_2  {+}\alpha_2^{-} \right)\\{+}\alpha_1  \beta_2  \lambda_2  \mu_1  \left(\lambda_1e  {-} \mu_2f  \right)^{2} \left(\mu_1  {+}\beta_1^{-} \right) \left(\lambda_2  {+}\alpha_2^{-} \right)
\end{multline*}
\end{proposition}
\begin{proof}
An equilibrium point $(x_1^*,x_3^*,a_1^*,b_1^*)$  satisfies the relations
\begin{align*}
\mu_2(f{-}b_1^*){+}\alpha_1^-a_1^*&=(\lambda_1{+}\alpha_1^-)a_1^*\\ &=\frac{\alpha_1x_1^*}{\alpha_1 x_1^*+  \alpha_2 (1{-}e{-}f{-}x_1^*{-}x_3^*)}((\lambda_1{+}\alpha_1^-) a_1^*+  (\lambda_2 {+}\alpha_2^-)(e{-}a_1^*)),\\
\lambda_2 (e-a_1^*){+}\beta_1^-b_1^*&=(\mu_1{+}\beta_1^-)b_1^*\\ &=\frac{\beta_1x_3^*}{\beta_1 x_3^*+  \beta_2 (1{-}e{-}f{-}x_1^*{-}x_3^*)}((\mu_1{+}\beta_1^-) b_1^*+  (\mu_2{+}\beta_2^-) (f{-}b_1^*)),
\end{align*}
the expressions for $a_1^*$ and $b_1^*$ follow directly, it is easily checked that  under Conditions~\eqref{CondM}, we have $a_1^*{\in}(0,e)$ and $b_1^*{\in}(0,f)$.

With some calculations we obtain the expressions for  $x_1^*$ and  $x_3^*$.  Conditions~\eqref{CondM} give that $x_1^*$, $x_3^*$ and $D$ are positive,  and, with  the relation
 \begin{multline*}
1{-}e{-}f{-}x_1^*{-}x_3^*\\=\frac{1}{D}\alpha_1  \beta_1  \lambda_1  \mu_1  \left(1{-}e {-}f\right) \left(\lambda_1e  {-}\mu_2f  \right)\left( \mu_1f{-}\lambda_2 e \right) \left(\mu_2  {+}\beta_2^{-} \right) \left(\lambda_2  {+}\alpha_2^{-} \right) >0,
 \end{multline*}
we conclude that the  element $(x_1^*,x_3^*,a_1^*,b_1^*)$ is  an equilibrium point  of the dynamical system $(x_1(t),x_3(t),a_1(t),b_1(t))$ in the state space ${\cal E}_P$.
\end{proof}
We now formulate a consequence of a classical useful analytic result on the location of the  roots of a polynomial of degree~4. It will be sufficient to establish our main stability results. See~\citet{Parks}, \citet{Gantmacher}, or~\citet{Fadali} for example. 

\begin{proposition}[Routh-Hurwitz Criterion]\label{RHCrit}
If $P(x)$ is a polynomial,
\[
P(x)=C_0x^4+C_1x^3+C_2x^2+C_3x+C_4,
\]
the Routh-Hurwitz coefficients $(R_i,1{\le}i{\le}4)$ of $P(x)$ are given by
\[
R_1=C_0, R_2=C_1, R_3=C_4,\quad
R_4=\frac{C_2C_1{-}C_0C_3}{C_1},\quad 
R_5=R_4C_3{-}C_4C_1.
\]
\begin{enumerate}
\item If these coefficients are all positive then all the roots of $P$ have a negative real part.
\item Otherwise, there is at least  a root of $P(x)$ with a positive real part.
\end{enumerate}
\end{proposition}
There is a slight abuse of notation, strictly speaking, the Routh-Hurwitz coefficients are determinants of matrices. See~\citet{Gantmacher}. The coefficients $R_4$ and $R_5$ differ from a (positive) multiplicative constant. We can now state our main stability result of this case. 

\begin{proposition}[Stability of the Equilibrium]\label{StabM1M}
Under the condition $e{+}f {<} 1$ and if, $\alpha_i^-{=}\beta_i^-{=}c$, for all $i{=}1$, $2$, for  $c{\in}\{0,1\}$,  then
  \begin{enumerate}
  \item[(S)] If $\mu_1f{-}\lambda_2e >0$ and $\lambda_1e{-}\mu_2f>0$,\\
    the equilibrium  $v_*$ of Relation~\eqref{EquiM1M} of the dynamical system~\eqref{ODEM} is stable.
  \item[(U)] If $\mu_1f{-}\lambda_2e<0$ and $\lambda_1e{-}\mu_2f<0$,\\
    the equilibrium  $v_*$ is unstable.
  \end{enumerate}
\end{proposition}
\begin{proof}
The principle of the proof of stability is fairly simple: if $A_*$ is the Jacobian matrix of the dynamical system~\eqref{ODEM} at the equilibrium point $v_*$, one has to prove that all its eigenvalues have a negative real part. Proposition~\ref{RHCrit} is used for the characteristic  polynomial $P_*$ of $A_*$. The coefficients of $P_*$ are rational functions of the reaction rates and of the parameters $e$ and $f$. In the following we will make a slight abuse of notation, by denoting $P_*(x)$ as the numerator of $P_*(x)$. This clearly does not change  the location of the roots of the the characteristic  polynomial. 

The computations in this proof were performed by using the symbolic computation software Maple\texttrademark~\cite{Maple}.  See below for the details. The restriction of the possible values of $\alpha_i^-$ and $\beta_i^-$ for $i{=}1$, $2$, is mainly due to the computational complexity of the algebraic expressions for the Rough-Hurwitz coefficients. 

For example, when $\alpha_i^-{=}\beta_i^-{=}0$ for $i{=}1$, $2$, the coefficient of $x^2$ of the numerator of  $P_*(x)$ is a polynomial of ten variables, $e$, $f$, $\alpha_i$, $\beta_i$, $\lambda_i$, $\mu_i$, $i{=}1$, $2$,  of total degree 37, with more than 1450 terms. Another difficulty is of being able to use the conditions of the signs of $1{-}e{-}f$, $\mu_1f{-}\lambda_2e$ and $\lambda_1e{-}\mu_2f$ of the proposition, to derive the signs of the five Rough-Hurwitz coefficients of Proposition~\ref{RHCrit}.

The main ingredient of this proof is the introduction of a convenient parametrization of the constants $e$, $f$ and $\lambda_1$, $\lambda_2$ and a careful handling of the corresponding algebraic expressions. 

\medskip
\noindent
{\bf Parametrization scheme.}
\begin{enumerate}
\item The condition $e{+}f{<}1$ is expressed by the fact that there exist $u$ and $v{>}0$ such that,
\[
e=\frac{u}{v+u(1+v)},\quad f=\frac{uv}{v+u(1+v)}.
\]
\item The condition~(S), i.e. $\mu_1f{-}\lambda_2e >0$ and $\lambda_1e{-}\mu_2f>0$, by the existence of $a$ and $b{>}0$ such that,
\[
\lambda_2=\frac{\mu_1f}{e(1+a)},\quad 
\lambda_1=\frac{\mu_2f(1+b)}{e}.
\]
\item The condition~(U), i.e. $\mu_1f{-}\lambda_2e <0$ and $\lambda_1e{-}\mu_2f<0$, by the existence of $a$ and $b{>}0$ such that,
\[
\lambda_2=\frac{\mu_1f(1{+}b) }{e},\quad 
\lambda_1=\frac{\mu_2 f}{e(1{+}a)}.
\]
\end{enumerate}
In the above example when $\alpha_i^-{=}\beta_i^-{=}0$ for $i{=}1$, $2$, the coefficient of $x^2$ for $P_*(x)$ is this time is a polynomial of the variables $a$, $b$, $u$, $v$, $\alpha_i$, $\beta_i$, $\mu_i$, $i{=}1$, $2$, of total degree 27 with 521 terms. The gain of complexity is not really  significant but our conditions are thus encoded by the fact that $a$, $b$, $u$ and $v$ are positive.  It will be crucial to assess the signs of the Rough-Hurwitz coefficients, expressed with polynomials of these positive variables.

Due to the sizes of the algebraic expressions, the results have been gathered in the files in the ancillary section of the arXiv submission~\citet{LR26} of the paper. The main Maple\texttrademark{} code is in the file {\tt Code.maple}. See also the last section of the appendix. The  names of the files for the \LaTeX{} expressions the five Rough-Hurwitz coefficients are in the general format {\tt X-Y.tex}, with $X{\in}\{0,1\}$ and $Y{\in}\{S,U\}$,
\begin{itemize}
\item $X{\in}\{0,1\}$, for the case $\alpha_i^-{=}\beta_i^-{=}X$ for $i{=}1$, $2$;
\item $Y{\in}\{\text{S,U}\}$ refers to the condition (S) or (U).
\end{itemize}
The exact expressions of some of the $R_i$, $1{\le}i{\le}5$, have thousands terms so that the sign of them may be  difficult to assess. For example, the coefficient $R_4$ in  {\tt 0-S.tex} has 4440 terms. For this reason, the \LaTeX{} file {\tt X-Y.tex} has been processed into a file {\tt X-Y-trace.txt} by removing all the positive variables  $a$, $b$, $u$, $v$, $\alpha_i$, $\beta_i$, $\mu_i$, $i{=}1$, $2$, and their powers. The result is a skeleton composed of the symbols of the set \{+,-,(, )\}, which makes the verification of signs easier.

For  {\tt X}${\in}\{0,1\}$, it is easily checked that, under condition~(S), all Rough-Hurwitz coefficients are positive in {\tt X-S-trace.txt}. Under condition~(U),  the coefficients $R_1$, $R_2$, $R_4$ and $R_5$ are positive and $R_3$ is negative  in {\tt X-U-trace.txt}. This concludes the proof of the proposition. 
\end{proof}

Proposition~\ref{StabProp} gives readily the following stability property for $H$.
\begin{corollary}
Under the conditions of Theorem~\ref{ThM1M}, if $\mu_1f{-}\lambda_2e{>}0$ and $\lambda_1e{-}\mu_2f{>}0$, then the subset $H{=}\{{ A}, { B}\}$ is stable for the sequence of Markov processes $(V_N(t))$. 
\end{corollary}
\section{Saturated Case, $e{+}f{>}1$}\label{SatSec}
In this last case, there are more enzyme than substrat. We study the possible equilibrium properties of the set of states for which there are few copies of substrat species, i.e. $H{=}\{{ S}_1,{ S}_2,{ S}_3\}$. 
\subsection{Few Substrat}
\addcontentsline{toc}{section}{ \ \thesubsection\quad Few Substrat}

\begin{theorem}\label{M2Th}
Under the condition $e{+}f{>}1$, if $H{=}\{{ S}_1,{ S}_2,{ S}_3\}$ and the initial conditions are 
\[
V_N^H(0)=v^H=(x_1,x_2,x_3)\in\N^3, \quad V_N^{H^c}(0)=(u_A^N,a_1^N,a_2^N,u_B^N,b_1^N,b_2^N)\in\N^6
\]
such that
\[
\lim_{N\to+\infty}\frac{V_N^{H^c}(0)}{N}= (e{-}{a}_1{-}a_2^0,a_1^0,a_2^0,f{-}b_1^0{-}b_2^0,b_1^0,b_2^0),
\]
with $b_2^0\steq{def}1{-}a_1^0{-}a_2^0{-}b_1^0$, and $(a_1^0,a_2^0,b_1^0){\in}{\cal E}_P$, 
\[
{\cal E}_P=\{ (a_1,a_2,b_1){\in}(0,{+}\infty)^3: a_1{+}a_2{+}b_1<1, 1{-}f{<}{a}_1{+}a_2< e\} 
\]
 then there exists $t_0{>}0$ such that, for the convergence in distribution, 
\begin{multline}\label{VM2}
\lim_{N\to+\infty}\left(\left(\frac{V_N^{H^c}(t)}{N},t{<}t_0\right),\Lambda_N^{H,t_0}\right)\\=\left(\left({u}_A(t),{a}_1(t),{a}_2(t),{u}_B(t),{b}_1(t),{b}_2(t), t{<}t_0\right),\Lambda_\infty^{t_0}\right)\\
{\steq{def}\left(\left(e{-}{a}_1(t){-}{a}_2(t),{a}_1(t),{a}_2(t),f{+}{a}_1(t){+}{a}_2(t){-}1,{b}_1(t),1{-}{a}_1(t){-}{a}_2(t){-}{b}_1(t), t{<}t_0\right),\Lambda_\infty^{t_0}\right)}
\end{multline}
where
\begin{enumerate}
\item the function $(a_1(t),a_2(t),b_1(t))$ is the solution of the ODE
\begin{equation}\label{ODEM2}
\begin{cases}
\dot{a}_1(t)&= \mu_2b_2(t)-\lambda_1 a_1(t),\\
\dot{a}_2(t)&=\alpha_2u_A(t)m_{X_2}(t) {-}(\lambda_2{+}\alpha_2^-)a_2(t),\\
\dot{b}_1(t)&=\lambda_2a_2(t)-\mu_1b_1(t),
\end{cases}
\end{equation}
with initial state $(a_1^0,a_2^0,b_1^0)$, with, for $t{\ge}0$,
\[
\begin{cases}
\displaystyle m_{X_1}(t)=\frac{\mu_2b_2(t){+}\alpha_1^-a_1(t)}{\alpha_1u_A(t)}, \qquad
m_{X_3}(t)=\frac{\lambda_2a_2(t){+}\beta_1^-b_1(t)}{\beta_1u_B(t)},\\
\displaystyle m_{X_2}(t)=\frac{\lambda_1a_1(t){+}\alpha_2^-a_2(t){+}\mu_1b_1(t){+}\beta_2^- b_2(t)}{\alpha_2 u_A(t){+}\beta_2u_B(t)}.
\end{cases}
\]
\item The measure $\Lambda_\infty$ is defined by, for $F{\in}{\cal C}_c(\R_+{\times}\N^3)$, 
\begin{multline*}
  \croc{\Lambda_\infty,F}=\int_{\R_+{\times}\N^3} F(s,x_1,x_2,x_3) \\ \Pois{m_{X_1}(s)}(\diff x_1)\Pois{m_{X_2}(s)}(\diff x_2)\Pois{m_{X_3}(s)}(\diff x_3)\diff s,
\end{multline*}
 almost surely,
\end{enumerate}
\end{theorem}
\begin{proof}
The proof is analogous, even simpler, to the proof of Theorem~\ref{ThM1L}. It is therefore skipped. 
It should be noted that the $O(1)$ processes are also all ``fast''  in this case. Their transitions rates are, at time $t{=}0$,
\[
v^H\longrightarrow v^H{+}
\begin{cases}
\scriptstyle{e_{X_1}}, & \scriptstyle{\alpha_1^-a_1^N+\mu_2b_2^N},\\
\scriptstyle{{-}e_{X_1}}, &\scriptstyle{\alpha_1u_A^Nx_1},\\
\scriptstyle{e_{X_2}}, &\scriptstyle{\mu_1b_1^N{+}\alpha_2^-a_2^N{+}\beta_2^-b_2^N{+}\lambda_1a_1^N},
    \end{cases}\qquad 
    \begin{cases}
      \scriptstyle{{-}e_{X_2}}, &\scriptstyle{\left(\alpha_2u_A^N{+}\beta_2u_B^N\right)x_2},\\
\scriptstyle{e_{X_3}}, &\scriptstyle{\lambda_2a_2^N{+}\beta_1^-b_1^N},\\
\scriptstyle{{-}e_{X_3}}, & \scriptstyle{\beta_1u_B^Nx_3}.
\end{cases}
\]
\end{proof}

The following lemma gives a characterization of the potential equilibrium points of the dynamical system of Theorem~\eqref{M2Th}. 
\begin{lemma}
If $(a_1^*,a_2^*,b_1^*,b_2^*){\in}(0,{+}\infty)^4$ is an equilibrium point of the  dynamical system  $(a_1(t),a_2(t),b_1(t),b_2(t))$ of Relation~\eqref{VM2}, then
\begin{equation}\label{abb}
a_2^*=\frac{\mu_1  \left(\mu_2 -(\lambda_1  + \mu_2)a_1   \right)}{\mu_2  \left(\lambda_2  +\mu_1  \right)},
\quad b^*_1=\frac{\lambda_2  \left(\mu_2-(\lambda_1 +\mu_2)a_1 \right)}{\mu_2  \left(\lambda_2  +\mu_1  \right)},\quad
b_2^*=\frac{\lambda_1}{\mu_2} a_1^*,
\end{equation}
where $a_1^*$ is a  root of the polynomial $P(x)$, for $x{\in}\R$, 
\begin{multline}\label{PolyP}
P(x)\steq{def}  (\lambda_1\mu_1{-}\lambda_2\mu_2)(\alpha_2\lambda_1\mu_2 (\lambda_2{+}\mu_1){-}\beta_2\lambda_2\mu_1(\lambda_1{+}\mu_2))x^2\\
{+}\mu_2\left[\rule{0mm}{5mm}\alpha_2\lambda_1\mu_2(\lambda_2{+}\mu_1)((\lambda_2{+}\mu_1)e{-}\mu_1){+}\beta_2\lambda_2\mu_1\left((\lambda_1{+}\mu_2)((\lambda_2{+}\mu_1)f{-}\lambda_2){+}\lambda_1\mu_1{-}\lambda_2\mu_2\right)\rule{0mm}{5mm}\right]x\\
  +\beta_2  \lambda_2  \mu_1 \mu_2^{2} (\lambda_2{-}(\lambda_2{+}\mu_1)f),
\end{multline}
furthermore
\begin{equation}\label{sumC}
a_1^*+a_2^*=\frac{\mu_1\mu_2+a_1^*(\lambda_2\mu_2-\lambda_1\mu_1)}{\mu_2 (\lambda_2  +\mu_1)},\quad
b_1^*+b_2^*=\frac{\lambda_2  \mu_2+a_1^*(\lambda_1\mu_1-\lambda_2\mu_2)}{\mu_2  \left(\lambda_2  +\mu_1  \right)}.
\end{equation}
\end{lemma}
The identification of equilibrium states is therefore reduced to the analysis of the roots of  a polynomial of degree~$2$. This looks quite simple, but as in Section~\ref{M1MSec},  the dependence of the coefficients of this polynomial on the parameters of the CRN, $e$, $f$, $\lambda_i$, $\mu_i$, $\alpha_i$, $\beta_i$, $i{=}1$, $2$, complicates significantly the investigation. This is discussed in the following proposition. 

\begin{proposition}[Equilibrium of the dynamical system of Relation~\eqref{VM2}]\label{M2ExFP}\ \\
Under the condition $e{+}f>1$,
\begin{enumerate}
\item if $\lambda_1\mu_1{-}\lambda_2\mu_2>0$.

\medskip
\noindent  If $e{\not=}\mu_2/(\lambda_1{+}\mu_2)$ and $f{\not=}\lambda_2/(\lambda_2{+}\mu_1)$,
there exists an equilibrium point if and only if the conditions
\begin{equation}\label{Cond21}
e > \frac{\mu_2}{\lambda_1{+}\mu_2} \text{ and }   f >\frac{\lambda_2}{\lambda_2{+}\mu_1},
\end{equation}
hold and in this case, it is unique. 
\item If $\lambda_1\mu_1{-}\lambda_2\mu_2<0$.\medskip
  \begin{enumerate}
  \item If the conditions
\begin{equation}\label{Cond2}
e > \frac{\mu_1}{\lambda_2{+}\mu_1} \text{ and }   f >\frac{\lambda_1}{\lambda_1{+}\mu_2},
\end{equation}
and 
\begin{equation}\label{Cond3}
\left(e{-}\frac{\mu_2}{\lambda_1{+}\mu_2}\right)\left(f {-}\frac{\lambda_2}{\lambda_2{+}\mu_1}\right) > 0,
\end{equation}
hold, then there exists a unique equilibrium point.
\item If Condition~\eqref{Cond2}
and relation
\begin{equation}\label{Cond32}
\left(e{-}\frac{\mu_2}{\lambda_1{+}\mu_2}\right)\left(f {-}\frac{\lambda_2}{\lambda_2{+}\mu_1}\right) < 0,
\end{equation}
hold, then there are  either $0$ or $2$ equilibrium points: There exists a subset $S$, resp. $S_0$, of $\R_+^2$ with non-empty interior such that if $(\alpha_2,\beta_2){\in}S$, resp. $(\alpha_2,\beta_2){\in}S_0$, then there are two equilibrium points, resp. there does not exist an equilibrium point. 
\end{enumerate}
\end{enumerate}
\end{proposition}
\begin{proof}
With Relation~\eqref{abb}, a possible equilibrium point $(a_1^*,a_2^*,b_1^*,b_2^*)$ should satisfy the relations
\begin{equation}\label{feq1}
0{<} a_1^*{<} C_0\steq{def}\frac{\mu_2}{\lambda_1  + \mu_2}, \quad a_1^*{+}a_2^*{<}e \text{ and } b_1^*{+}b_2^*{<}f,
\end{equation}
and therefore, with Relations~\eqref{sumC}, the inequalities
\begin{equation}\label{feq12}
\frac{\mu_1}{\lambda_2  + \mu_1}{-}e < a_1^*\frac{\lambda_1\mu_1{-}\lambda_2\mu_2}{\mu_2(\lambda_2  + \mu_1)}
< f{-} \frac{\lambda_2}{\lambda_2  +\mu_1}.
\end{equation}
As it can be seen this implies that the relation $e{+}f{>}1$ holds. 

We define
\[
C_1\steq{def}\mu_2\frac{\mu_1{-}(\lambda_2{+}\mu_1)e}{\lambda_1\mu_1{-}\lambda_2\mu_2} \text{ and } C_2\steq{def}\mu_2\frac{(\lambda_2{+}\mu_1)f{-} \lambda_2}{\lambda_1\mu_1{-}\lambda_2\mu_2},
\]
note that
\begin{equation}\label{Ineq}
\begin{cases}
&C_2{-}C_1{=}\displaystyle\frac{\mu_2(\lambda_2{+}\mu_1)(e{+}f{-}1)}{\lambda_1\mu_1{-}\lambda_2\mu_2}, \quad C_2{-}C_0{=}\frac{\mu_2  \left(\lambda_2  +\mu_1  \right) \left((\lambda_1{+}\mu_2)f{-} \lambda_1 \right)}{\left(\mu_2  +\lambda_1  \right) \left(\lambda_1  \mu_1{-}\lambda_2  \mu_2  \right)},\\
\ \\
 &C_0{-}C_1{=}\displaystyle\frac{\mu_2(\lambda_2{+}\mu_1)((\lambda_1{+}\mu_2)e{-}\mu_2)}{(\lambda_1\mu_1{-}\lambda_2\mu_2)(\mu_2{+}\lambda_1)}.
\end{cases}
\end{equation}
With Relation~\eqref{PolyP} (and trite calculations), we obtain
\begin{equation}\label{RelP}
\begin{cases}
&P(0)\displaystyle=\beta_2  \lambda_2  \mu_1 \mu_2^{2} (\lambda_2{-}(\lambda_2{+}\mu_1)f )\smallskip\\
&P\left(C_0\right)\displaystyle=\mu_2^{3} \alpha_2  \lambda_1  \left(\lambda_2 {+}\mu_1  \right)^{2} \left((\lambda_1{+}\mu_2)e  {-}\mu_2  \right)\smallskip\\
&P(C_1)\displaystyle=\mu_1 \mu_2^{2} \beta_2  \lambda_2  \left(\lambda_2{+}\mu_1  \right)^{2} \left(\lambda_1  \mu_1  {-}\lambda_2  \mu_2  \right) \left(\mu_2{-}(\lambda_1{+} \mu_2)e\right) \left(e{+}f{-}1\right)\smallskip\\
&P(C_2)\displaystyle=\mu_2^{3} \alpha_2  \lambda_1\left(\lambda_1  \mu_1  {-}\lambda_2  \mu_2\right)  \left(\lambda_2 {+}\mu_1  \right)^{2} \left((\lambda_2{+}\mu_1)f  {-}\lambda_2  \right) \left(e{+}f{-}1\right).
\end{cases}
\end{equation}

If $\lambda_1\mu_1-\lambda_2\mu_2>0$ holds, to have an equilibrium $(a_1^*,a_2^*,b_1^*,b_2^*)$ for the dynamical system, the root $a_1^*$ of the polynomial $P$ must be  such that  $a_1^*{\in}(C_1{\vee}0,C_2{\wedge}C_0)$ because of Relation~\eqref{feq12}.  In particular, we must have the relations $C_2{>}0$ and $C_1{<}C_0$, which imply, by using Inequalities~\eqref{Ineq}, that the relations
\[
f> \frac{\lambda_2}{\lambda_2{+}\mu_1} \text{ and } e >\frac{\mu_2}{\lambda_1{+}\mu_2}
\]
hold. With Relations~\eqref{RelP}, we obtain the inequalities  $P(C_1){<}0$  and $P(C_0){>}0$ therefore there is a unique equilibrium point $(a_1^*,a_2^*,b_1^*,b_2^*)$ in $\R_+^4$ with $a_1^*{\in}(C_1^+,C_0{\wedge}C_2)$. The first case is proved.

For the second part of the proposition, when $\lambda_1\mu_1{-}\lambda_2\mu_2{<}0$, we must have $a_1^*{\in}(C_2^+,C_1{\wedge}C_0)$ and therefore $C_1{>}0$ and $C_2{<}C_0$, and hence the conditions
\[
e>\frac{\mu_1}{\lambda_2{+}\mu_1} \text{ and }f>\frac{\lambda_1}{\lambda_1{+}\mu_2}.
\]

From now on, assume that Condition~\eqref{Cond2} holds.

If
\[
e> \frac{\mu_2}{\lambda_1{+}\mu_2} \text{ and } f> \frac{\lambda_2}{\lambda_2{+}\mu_1},
\]
then  $C_2{<}0{<}C_0{<}C_1$ and by using again Relations~\eqref{RelP}, we have $P(0){<}0$ and $P(C_0){>}0$. Hence, there is a unique equilibrium point $(a_1^*,a_2^*,b_1^*,b_2^*)$ in $\R_+^4$ with $a_1^*{\in}(0,C_1)$. 

Otherwise if
\[
e< \frac{\mu_2}{\lambda_1{+}\mu_2} \text{ and } f< \frac{\lambda_2}{\lambda_2{+}\mu_1},
\]
then,  $0{<}C_2{<}C_1{<}C_0$ and $P(C_2){>}0$ and $P(C_1){<}0$. Hence, there is a unique equilibrium point $(a_1^*,a_2^*,b_1^*,b_2^*)$ in $\R_+^4$ with $a_1^*{\in}(C_2,C_1)$. 

For the proof of the last part of the proposition, we introduce $W_0$ the location of the extremum of $P$,
 \[
 W_0=-\frac{\mu_2\left[\rule{0mm}{5mm}\alpha_2\lambda_1\mu_2(\lambda_2{+}\mu_1)((\lambda_2{+}\mu_1)e{-}\mu_1){+}\beta_2\lambda_2\mu_1\left(\rule{0mm}{4mm}(\lambda_1{+}\mu_2)((\lambda_2{+}\mu_1)f{-}\lambda_2){+}\lambda_1\mu_1{-}\lambda_2\mu_2\right)\rule{0mm}{5mm}\right]}{2(\lambda_1\mu_1{-}\lambda_2\mu_2)(\alpha_2\lambda_1\mu_2 (\lambda_2{+}\mu_1){-}\beta_2\lambda_2\mu_1(\lambda_1{+}\mu_2))}.
 \]
With some calculations, we have
 \[
 C_0{-}W_0=\frac{\alpha_2\lambda_1\mu_2((\lambda_2{+}\mu_1)((\lambda_1{+}\mu_2)e{-}\mu_2)+\lambda_1\mu_1{-}\lambda_2\mu_2)+ \beta_2\lambda_2\mu_1(\mu_2{+}\lambda_1)((\lambda_1{+}\mu_2)f{-}\lambda_1)}{2(\lambda_1\mu_1{-}\lambda_2\mu_2)(\alpha_2\lambda_1\mu_2 (\lambda_2{+}\mu_1){-}\beta_2\lambda_2\mu_1(\lambda_1{+}\mu_2))}
 \]
 and
  \[
  W_0{-}C_2= -(\lambda_2{+}\mu_1)\mu_2
  \frac{\alpha_2((\lambda_2{+}\mu_1)(e{+}f{-}1){+}(\lambda_2{+}\mu_1)f{-}\lambda_2){+}\beta_2\lambda_2\mu_1(\lambda_1{-}(\lambda_1{+}\mu_2)f)}{2(\lambda_1\mu_1{-}\lambda_2\mu_2)(\alpha_2\lambda_1\mu_2 (\lambda_2{+}\mu_1){-}\beta_2\lambda_2\mu_1(\lambda_1{+}\mu_2))}
  \]

If Condition~\eqref{Cond2} and
\[
e>\frac{\mu_2}{\lambda_1{+}\mu_2} \text{ and } f< \frac{\lambda_2}{\lambda_2{+}\mu_1}
\]
hold, we can write
\[
f=w\frac{\lambda_1}{\lambda_1+\mu_2} +(1-w)\frac{\lambda_2}{\lambda_2{+}\mu_1},
\]
with $w{\in}(0,1)$. We have $0{<}C_2{<}C_0{<}C_1$ and $P(C_2){>}0$ and $P(C_0){>}0$. Hence there are either $0$ or $2$ roots (counting the multiplicity if $P(W_0){=}0$), of $P$ in $(C_2,C_0)$. 

It is easily checked that 
\[
\lim_{(\alpha_2,\beta_2)\to(0,1)} (W_0{-}C_2,C_0{-}W_0)=(1{-}w)\frac{\mu_2}{2(\mu_2{+}\lambda_1)}(1,1)
\]
and, with some calculations, 
\[
\lim_{(\alpha_2,\beta_2)\to(0,1)} P(W_0)=\mu_2^2\mu_1^3\lambda_2^3(1{-}w)^2(\mu_2{+}\lambda_1)^2(\lambda_1\mu_1{-}\lambda_2\mu_2)^3 <0. 
\]

Hence there exists a neighborhood $V$ of $(1,0)$ such that if $(\alpha_2,\beta_2){\in}V$, then $W_0{\in}(C_2,C_0)$ and $P(W_0)<0$. In this case there are two roots of $P$ in $(C_2,C_0)$.

The result
\[
\lim_{(\alpha_2,\beta_2)\to(1,0)} P(W_0)=\mu_2^5\lambda_1^3(\lambda_2{+}\mu_1)^3(\lambda_1{+}\mu_2)(\lambda_2\mu_2{-}\lambda_1\mu1)((\lambda_2{+}\mu_1)e{-}\mu_1)^2>0,
\]
shows that, for $(\alpha_2,\beta_2)$ in a small neighborhood of $(1,0)$: if $W_0{\in}(C_2,C_1)$ then there does not exist a root of $P$ in $(C_2,C_1)$, otherwise, if $W_0{\not\in}(C_2,C_1)$ the same conclusion holds obviously.

If Condition~\eqref{Cond2} and
\[
e<\frac{\mu_2}{\lambda_1{+}\mu_2} \text{ and } f>\frac{\lambda_2}{\lambda_2{+}\mu_1}
\]
we can write
\[
e=w\frac{\mu_2}{\lambda_1+\mu_2} +(1-w)\frac{\mu_1}{\lambda_2{+}\mu_1},
\]
with $w{\in}(0,1)$. We have $C_2{<}0{<}C_1{<}C_0$ and $P(0){<}0$ and $P(C_1){<}0$. There is either $0$ or $2$ roots of $P$ in $(0,C_1)$. 

\[
\lim_{(\alpha_2,\beta_2)\to(1,0)} (W_0,C_1{-}W_0)=w\frac{\mu_2}{2(\mu_2{+}\lambda_1)}(1,1)
\]
and
\[
\lim_{(\alpha_2,\beta_2)\to(1,0)} P(W_0)=(\lambda_2{+}\mu_1)^3(\lambda_2\mu_2{-}\lambda_1\mu_1)^3\mu_2^5\lambda_1^3w^2>0,
\]
there are two roots of $P$ in $(0,C_1)$ for $(\alpha_2,\beta_2)$ in a small neighborhood of $(1,0)$.

Finally, the relation
\[
\lim_{(\alpha_2,\beta_2)\to(0,1)} P(W_0)=\lambda_2^3\mu_1^3\mu_2^2(\lambda_2{+}\mu_1)^2(\lambda_1{+}\mu_2)^3(\lambda_1\mu_1{-}\lambda_2\mu_2)((\lambda_1{+}\mu_2)f{-}\lambda_1)^2<0,
\]
shows that there does not exist a root of $P$ in $(0,C_1)$ for $(\alpha_2,\beta_2)$ in a small neighborhood of $(0,1)$.

The proposition is proved.
\end{proof}
\printbibliography
\appendix
\section{Appendix}
\subsection{End of the Proof of Theorem~\ref{ThM1L}}\label{O1SSec}
\begin{proof}
The variables  $(X_3^N(t),A_2^N(t),B_1^N(t))$ are slow in the sense that the transition rates of positive jumps are $O(1)$. Their transitions are given by, at time $t{=}0$, 
\begin{equation}\label{O1SHEq2}
v{\longrightarrow}   v{+} \begin{cases}
\scriptstyle{e_{A_2}}, &\scriptstyle{\alpha_2x_2u_A},\\
\scriptstyle{{-}e_{A_2}}, &\scriptstyle{\alpha_2^-a_2},\\
\scriptstyle{{-}e_{B_1}}, &\scriptstyle{\mu_1b_1},\\
\scriptstyle{e_{B_1}{-}e_{A_2}}, &\scriptstyle{\lambda_2a_2},
\end{cases}
\quad
\begin{cases}
\scriptstyle{e_{X_3}}, &\scriptstyle{\lambda_2a_2{+}\beta_1^-b_1},\\
\scriptstyle{{-}e_{X_3}}, &\scriptstyle{\beta_1a_2u_B^N},
\end{cases}
\end{equation}

We will give a sketch of the proof. The missing arguments, like in the proof of a standard stochastic averaging principle, see~\citet{Kurtz1992}, are straightforward. 

Relation~\eqref{CoupO1Eq} of Section~\ref{CoupO1Sec} show that  the sequence of the  total number of jumps of these random variables in the time interval $[0,t_0)$ is tight.
The process  $(U_B^N(t)/N,t{<}t_0)$ is converging to a non-trivial process. The fact that the negative jumps of $(X_3^N(t),t{<}t_0)$ occur at a rate proportional to $U_B^N(t)$ and that the rates of positive jumps are bounded  gives that, when $X_3^N(t){>}0$, some $t{>}0$, this variable becomes $0$ right away.  This implies that, for {\em the occupation measure killed at time} $t_0$,  the process $(X_3^N(t),t{<}t_0)$ is  identically zero. This is reflected in the transitions of $(L_1(t),L_2(t))$:  a ${ AS}_2$ species is transformed directly in a ${ BS}_3$ species at rate $\lambda_2$. We have thus only to consider the process $(R^S_N(t)){=}(A_2^N(t),B_1^N(t))$.  Let $(s_k^N)$ the sequence of instants of jumps of $(R^S_N(t))$.

To conclude the proof of the theorem it is enough to show the convergence in distribution, for $G{\in}{\cal C}_c(\R_+{\times}\N^4)$, 
\begin{multline}\label{SReq0}
  \lim_{N\to+\infty} \left(\int_{s_k^N{\wedge}t_0}^{s_{k+1}^N{\wedge}t_0} G(u,(X_2^N(u),U_A^N(u), A_2^N(u),B_1^N(u)))\diff u, k{\ge}1\right)\\
  =  \left(\int_{s_k^\infty{\wedge}t_0}^{s_{k+1}^\infty{\wedge}t_0}\int_{\N^2} G(u,(x_2,u_A,R^S_\infty(u)))\pi_{C(u)}(\diff x_2,\diff u_A)\diff u,k{\ge}1\right),
\end{multline}
where $(R^S_\infty(t))$ is a Markov process on $\N^2$ with the transitions~\eqref{SlowTrans} and $(s_k^\infty)$ is the sequence of its instants of jumps.  

Since the Poisson processes ${\cal P}_{\kappa},\kappa{\in}\{\alpha_2,\alpha_2^-,\mu_1\}$, are independent of the other Poisson processes ${\cal P}_\kappa$, we have
  \[
  \P\left(s_1^N{\ge}t \mid {\cal F}_t\right)=
  \exp\left(-((\lambda_2+\alpha_2^-)a_2{+}\mu_1b_1)t-\alpha_2\int_0^t U_A^N(u)X_2^N(u)\diff u\right).
  \]
  This can be rephrased as follows: Let $E_1$ be an exponentially distributed random variable with parameter $1$ independent of ${\cal P}_\kappa$, $\kappa{\in}{\cal R}_P$, if
  \[
 \widetilde{s}_1^N=\inf\left\{t{>}0: ((\lambda_2+\alpha_2^-)a_2{+}\mu_1b_1)t{+}\alpha_2\int_0^t U_A^N(u)X_2^N(u)\diff u=E_1\right\},
  \]
  then $\widetilde{s}_1^N{\steq{dist}}s_1^N$. 

With the criterion of the modulus of continuity, see Theorem~7.3 of~\citet{Billingsley}, the convergence of the sequence of random measures $(\Xi_N^{t_0})$, see Relation~\eqref{XiN},  gives the following convergence in distribution 
\begin{multline*}
  \lim_{N\to+\infty}  \left(\int_0^{t} G(u,X_2^N(u), U_A^N(u))\diff u,t{<}t_0\right)
  \\=\left(\int_0^{t}\int_{\N^2} G(u,(x_2, u_A))\pi_{C(t)}(\diff x_2, \diff u_A)\diff u,t{<}t_0\right),
\end{multline*}
  for any $G{\in}{\cal C}_c(\R_+{\times}\N^4)$ and,  consequently,
\begin{multline}\label{SReq1}
  \lim_{N\to+\infty}  \left( \left(\int_0^{t} U_A^N(u)X_2^N(u))\diff u,\int_0^{t} G(u,X_2^N(u), U_A^N(u))\diff u\right),t{<}t_0\right)
\\  =\left(\left(r_L(t),\int_0^{t}\int_{\N^2} G(u,(x_2, u_A))\pi_{C(u)}(\diff x_2, \diff u_A)\diff u\right),t{<}t_0\right),
\end{multline}
where $(r_L(t))$ is defined by Relation~\eqref{SlowTrans}. 

Note that,
\begin{multline*}
  \int_{0}^{s_{1}^N{\wedge}t_0}\hspace{-6mm} G(u,(X_2^N(u),U_A^N(u), A_2^N(u),B_1^N(u)))\diff u
\\  = \int_{0}^{s_{1}^N{\wedge}t_0}\hspace{-6mm} G(u,(X_2^N(u),U_A^N(u), a_2,b_1))\diff u,
\end{multline*}
 Relation~\eqref{SReq1} gives the convergence in distribution,
\begin{multline}\label{SReq2}
 \lim_{N\to+\infty} \left( \int_{0}^{\widetilde{s}_{1}^N{\wedge}t_0}\hspace{-6mm} G(u,(X_2^N(u),U_A^N(u), a_2,b_1))\diff u,\right)
\\ = \left( \int_{0}^{S_1{\wedge}t_0}\int_{\N^2}\hspace{-3mm} G(u,(x_2,u_A, a_2,b_1))\pi_{C(u)}(\diff x_2,\diff u_A) \diff u,\right)
\end{multline}
 with
 \[
 S_1=\inf\left\{t{>}0: (\lambda_2a_2{+}\mu_1b_1)t{+}\alpha_2 r_L(t) =E_1\right\}\steq{dist} s_1^\infty.
 \]
 The right-hand side of Relation~\eqref{SReq2} has therefore the same distribution as
 \[
  \left( \int_{0}^{s_1^\infty{\wedge}t_0}\int_{\N^2}\hspace{-3mm} G(u,(x_2,u_A, R_\infty^S(u)))\pi_{C(u)}(\diff x_2,\diff u_A) \diff u,\right).
  \]
  We have thus proved the convergence~\eqref{SReq0} for $k=0$. From there we use the strong Markov property of the process $(V_N(t))$ for the stopping time $s_1^N$ and proceed as in the first step, by taking care of an additional aspect: The values $(X_2^N(s_1^N),U_A^N(s_1^N))$  are not, a priori, ``fixed'' as in the statement of the theorem, nevertheless we can use the same argument as in the proof of Proposition~\ref{StabProp} for the location of $V_N^H(s_1)$ by using a finite set $K_1$ and a convenient (small) stopping time. Based on these arguments, the proof can be concluded with an induction scheme on the sequence $(s_k^N)$.
\end{proof}
\subsection{Weight on the Right}\label{M1RSec}
This section is just a quick re-formulation of the results of Section~\ref{M1LSec} obtained by symmetry. 
\begin{theorem}\label{ThM1R}
Under the condition $e{+}f{<}1$, if $H{=}\{  { S}_1, { S}_2,  { AS_1},{ B},  { BS_2}\}$ and the initial condition is 
\[
V_N^H(0)=(x_1,x_2,a_1,u_B,b_2)\in\N^5, \quad V_N^{H^c}(0)=(x_3^N,u_A^N,a_2^N,b_1^N)\in\N^4,
\]
such that
\[
\lim_{N\to+\infty}\frac{V_N^{H^c}(0)}{N}=(1{-}f{-}a_2^0,e{-}a_2^0,a_2^0,f),
\]
with $a_2^0{\in}(0,e)$, then there exists $t_0{>}0$ such that,  for the convergence in distribution
\begin{multline}\label{CVWR}
\lim_{N\to+\infty}\left(\left(\left(\frac{v_N^{H^c}(t)}{N}\right),t{<}t_0\right),\Lambda^{H,t_0}_N\right)\\ =\left((v(t),t{<}t_0)),\Lambda^{t_0}_{\infty}\right)=
\left(\left(\left(1{-}f{-}{a}_2(t),e{-}{a}_2(t),{a}_2(t),f\right),t{<}t_0\right),\Lambda^{t_0}_\infty\right),
\end{multline}
where 
\[
a_2(t)=a_2^0\exp(-\lambda_2t)+\frac{\mu_1 f}{\lambda_2}\left(1-\exp(-\lambda_1t)\right),
\]
 and, for $F{\in}{\cal C}_c(\R_+{\times}\N^5)$, almost surely the relation
\begin{multline*}
\int_{\R_+{\times}\N^5} F(s,\underbrace{(x_1,x_2,a_1,u_B,b_2)}_{v^H})\Lambda_{\infty}(\diff s, \diff v^H)\\ =\int_{\R_+{\times}\N^3} F(s,(0,x_2,L_2(s),u_B,L_1(s)))\pi_{C(s)}(\diff x_2,\diff u_B)\diff s
\end{multline*}
holds, where, for $C{\in}(0,{+}\infty)^5$,  $\pi_{C}$ is the invariant measure of Proposition~\ref{MMNet}, and, for $t{\ge}0$, 
\[
C(t){=}\left(\rule{0mm}{5mm}\mu_1f,\alpha_2^-{a}_2(t),\alpha_2(e{-}{a}_2(t)),\beta_1^-f,\beta_1(1{-}f{-}{a}_2(t))\right)
\]
The stochastic processes $(L_1(t)),L_2(t))$ is associated to a non-homogeneous network of $M/M/\infty$ queues with initial state $(a_2,b_1)$, in state $\ell{=}(\ell_1,\ell_2)$ at time $t$, its transition rates are 
  \[
    \ell{\longrightarrow}   \ell {+} \begin{cases}
\scriptstyle{e_{1}}, &\scriptstyle{\beta_2r_R(t)},\\
\scriptstyle{{-}e_{1}}, &\scriptstyle{\beta_2^-\ell_1},\\
\scriptstyle{{-}e_{2}}, &\scriptstyle{\lambda_1\ell_2},\\
\scriptstyle{e_{2}{-}e_{1}}, &\scriptstyle{\mu_2\ell_1},
   \end{cases}
\qquad   \text{ with }   r_R(t)=\int_{\N^2} x_2u_B\pi_{C(t)}(\diff x_2,\diff u_B).
\]

\end{theorem}

\begin{corollary}
Under the conditions of Theorem~\ref{ThM1R} if $\mu_1f{<}\lambda_2 e$ then  the subset $H{=}\{  { S}_1, { S}_2,  { AS_1},{ B},  { BS_2}\}$ is stable for the sequence of Markov processes $(V_N(t))$. 
\end{corollary}
\subsection{Maple Code}\label{MapCod}

\begin{verbatim}
#General Maple Code

a2:=e-a1:
b2:=f-b1:
y:=1-e-f-x-z:
mA:=normal(((lambda1+alp1m)*a1+(lambda2+alp2m)*a2)/(alpha1*x+alpha2*y)):
mB:=normal(((mu1+bet1m)*b1+(mu2+bet2m)*b2)/(beta1*z+beta2*y)):

F1 := proc(x,z,a1,b1) normal(mu2*b2+alp1m*a1-alpha1*mA*x) end:
F2  := proc(x,z,a1,b1) normal(lambda2*a2+bet1m*b1-beta1*mB*z) end:
F3 :=  proc(x,z,a1,b1) normal(alpha1*mA*x-(lambda1+alp1m)*a1) end:
F4 :=  proc(x,z,a1,b1) normal(beta1*mB*z-(mu1+bet1m)*b1) end:

W:=solve({F1(x,z,a1,b1)=0,F2(x,z,a1,b1)=0,F3(x,z,a1,b1)=0,
                                    F4(x,z,a1,b1)=0},{x,z,a1,b1}):
ai := subs(W[1],a1): 
bi := subs(W[2],b1): 
xi := subs(W[3],x): 
zi := subs(W[4],z): 

# Jacobian matrix at equilibrium
with(LinearAlgebra):
with(VectorCalculus)

A:=Jacobian([F1(x,z,a1,b1), F2(x,z,a1,b1), F3(x,z,a1,b1),
               F4(x,z,a1,b1)],[x,z,a1,b1]=[xi,zi,ai,bi]):

# Change of variables (I)
B := map(normal, subs(e=u/(v+u*(1+v)),f=u*v/(v+u*(1+v)), A)):

# Change of variables (II)

#Underloaded case
lambda2:=mu1*f/e/(1+a):
lambda1:=mu2*(1+b)*f/e:

#overloaded case 
lambda2:=(1+b)*mu1*f/e:
lambda1:=mu2*f/e/(1+a):

# Characteristic Polynomial 
P := proc(x) map( factor,
          taylor( numer(CharacteristicPolynomial(B, x)),x)) end:

#  Coeficients of P
C0:=normal(coef(P(x), x, 4)):
C1:=normal(coef(P(x), x, 3)):
C2:=normal(coef(P(x), x, 2)):
C3:=normal(coef(P(x), x, 1)):
C4:=normal(coef(P(x), x,0)):

# Routh-Hurwitz Coeficients
R1:=C0:
R2:=C1:
R3:=C4:
R4:= normal((C2*C1 - C0*C3)/C1):
R5:= normal(R4*C3-C4*C1):
\end{verbatim}

\end{document}